\documentclass[11pt]{article}

\usepackage{amssymb,amsmath,amsthm}
\usepackage{bussproofs}
\usepackage{tikz}
\usepackage{comment}

\newcommand{\class}{\mathsf}
\newcommand{\set}[2]{\{ #1 \: | \linebreak[0] \: #2 \}}
\newcommand{\algebra}{\mathbf}
\newcommand{\logic}{\mathcal}

\newcommand{\B}{\logic{B}}
\newcommand{\CL}{\logic{CL}}
\newcommand{\LP}{\logic{LP}}
\newcommand{\K}{\logic{K}}
\newcommand{\Kleq}{\logic{K}^{\leq}}
\newcommand{\ETL}{\logic{ETL}}

\newcommand{\ECQ}{\logic{ECQ}}
\newcommand{\GentzenRelation}{G}
\newcommand{\GB}{\GentzenRelation\B}

\newcommand{\GentzenCalculus}{\mathbf{GC}}
\newcommand{\GKforB}{\GentzenCalculus\B}
\newcommand{\GKforCL}{\GentzenCalculus\logic{CL}}
\newcommand{\GKforLP}{\GentzenCalculus\logic{LP}}
\newcommand{\GKforK}{\GentzenCalculus\logic{K}}

\newcommand{\GKforETL}{\GentzenCalculus\logic{ETL}}
\newcommand{\GKforECQ}{\GentzenCalculus\logic{ECQ}}

\newcommand{\FmAlgebra}{\algebra{Fm}}
\newcommand{\LPmatrix}{\algebra{LP_{3}}}
\newcommand{\Kmatrix}{\algebra{K_{3}}}

\renewcommand{\Bmatrix}{\algebra{B_{4}}}
\newcommand{\ETLmatrix}{\algebra{ETL_{4}}}

\newcommand{\dmneg}{{-}}
\newcommand{\seq}{\vartriangleright}
\newcommand{\infleq}{\sqsubseteq}

\newcommand{\tautrans}{\boldsymbol{\tau}}
\newcommand{\rhotrans}{\boldsymbol{\rho}}

\newcommand{\SequentSet}{S}
\DeclareMathOperator{\AtSeqAux}{At}
\newcommand{\AtSeq}[1]{\AtSeqAux (#1)}

\DeclareMathOperator{\Log}{Log}

\newcommand{\HypSep}{.5in}
\newcommand{\LargerHypSep}{1.0in}

\newtheorem{theorem}{Theorem}[section]
\newtheorem{proposition}[theorem]{Proposition}
\newtheorem{lemma}[theorem]{Lemma}
\newtheorem{corollary}[theorem]{Corollary}

\newtheorem{definition}[theorem]{Definition}

\author{Adam P\v{r}enosil}
\title{Cut elimination, identity elimination, and interpolation in super-Belnap logics}
\date{}

\begin{document}

\maketitle

\begin{abstract}
  We develop a Gentzen-style proof theory for super-Belnap logics (extensions of the four-valued Dunn--Belnap logic), expanding on an approach initiated by Pynko. We show that just like substructural logics may be understood proof-theoretically as logics which relax the structural rules of classical logic but keep its logical rules as well as the rules of Identity and Cut, super-Belnap logics may be seen as logics which relax Identity and Cut but keep the logical rules as well as the structural rules of classical logic. A generalization of the cut elimination theorem for classical propositional logic is then proved and used to establish interpolation for various super-Belnap logics. In particular, we obtain an alternative syntactic proof of a refinement of the Craig interpolation theorem for classical propositional logic discovered recently by Milne.
\end{abstract}

\section{Introduction}

  The four-valued Dunn--Belnap logic $\B$ \cite{dunn66,dunn76,belnap77a,belnap77b} and its two three-valued extensions, the strong Kleene logic $\K$ \cite{kleene38,kleene52} and the Logic of Paradox $\LP$ \cite{priest79}, are well-known non-classical logics which have received a good deal of attention in the past decades, particularly from philosophically minded logicians. However, it is only in the last few years that a systematic study of the family of extensions of $\B$, called \emph{super-Belnap logics} by Rivieccio in \cite{rivieccio12}, was initiated. This paper is intended as a contribution to the investigation of this family of logics from a Gentzen-style proof-theoretic point of view.

  The starting point of this paper is Pynko's recent observation \cite{pynko10} that a Gentzen calculus for $\LP$ (for $\B$) may be obtained by adding the inverses of the logical introduction rules to a standard Gentzen calculus for classical logic and dropping the Cut rule (and the Identity axiom). Our main contribution is to extend this observation to arbitrary super-Belnap logics and formulate a useful generalization of the cut elimination theorem for classical logic to Gentzen calculi for super-Belnap logics. An immediate corollary of this result is then a broad sufficient condition for a super-Belnap logic to enjoy (a slight strengthening of) the Craig interpolation property, covering many of the new super-Belnap logics introduced in the last few years.

  In particular, we show that each super-Belnap logic corresponds to an extension of Pynko's calculus by a set of structural rules (rules which do not contain logical connectives). Super-Belnap logics may therefore be viewed as a class of logics analogical but orthogonal to the class of substructural logics. Substructural logics are obtained by keeping the logical rules of the classical Gentzen calculus fixed (as well as the Identity and Cut rules, although this is usually not mentioned explicitly) and tinkering with the structural rules of Exchange, Weakening, and Contraction. The situation with super-Belnap logics is dual: it is the rules of Exchange, Weakening, and Contraction which are kept fixed and the rules of Identity and Cut which are free to vary. In other words, we are justified in calling super-Belnap logics \emph{subreflexive and subtransitive logics} by analogy with substructural logics.

  Some of the results proved below may be of interest even to the classical logician. We recall Pynko's insight that the Logic of Paradox $\LP$ may be used to prove the admissibility of Cut in the Gentzen calculus for classical logic, and dualize it to prove the antiadmissibility of Identity. We then generalize the cut elimination procedure to proofs from non-empty sets of sequential premises, and use it to obtain a syntactic proof of the non-classical refinement of the Craig inter\-polation theorem for classical propositional logic $\CL$ discovered recently by Milne \cite{milne16}.

  The idea of using (non-deterministic) three-valued semantics to prove the admissibility of Cut dates back to the work of Sch\"{u}tte \cite{schuette60}. Later, it was used by Girard~\cite{girard87} to provide a three-valued semantics for a standard Gentzen calculus for classical logic without Cut. Dually, H\"{o}sli and J\"{a}ger~\cite{hosli+jager94} provided a three-valued semantics for a Gentzen calculus for classical logic without Identity. These ideas were then combined and extended by Lahav and Avron~\cite{lahav+avron13}, who provided a uniform way of defining a non-deterministic four-valued Kripke semantics for a wide range of Gentzen calculi without Cut or Identity or both. The difference between these approaches and the approach of Pynko~\cite{pynko10}, which we build on in the present paper, is that the latter is concerned with providing a semantics for a calculus which includes elimination rules (i.e.\ the inverses of introduction rules). Their presence will allow us to define what we call the analytic--synthetic normal form for proofs from a non-empty set of sequents. Note that the connection betwen deterministic semantics and the invertibility of logical rules was already pointed out in the two-valued case by Avron, Ciabattoni, and Zamansky~\cite{avron+ciabattoni+zamansky09}.

  It is worth recalling here that there are basically two distinct approaches to providing a Gentzen calculus for a given logic, in particular for~$\B$. In the first approach, the logic $\B$ is the logic of provable sequents. That is, $\Gamma \vdash_{\B} \varphi$ if and only if the sequent $\Gamma \seq \varphi$ is provable. This is the approach taken by Pynko \cite{pynko95} and Font \cite{font97}.\footnote{Pynko's calculus uses multiple-conclusion sequents and contains introduction rules for negated conjunction and disjunction but no structural rules (apart from Identity). The structural rules of Cut, Exchange, Weakening, and Contraction are then shown by Pynko to be admissible in this calculus. Font also briefly considers a single-conclusion version of this calculus with the structural rules present. However, the main calculus which he studies differs from Pynko's in replacing the introduction rules for negated conjunction and disjunction by Contraposition and Cut. Font's calculus, while being strongly adequate for~$\B$ in the sense of the theory of Font and Jansana~\cite{font+jansana09}, is in fact more of a calculus for the quasiequational theory of De Morgan lattices: De Morgan lattices form the equivalent algebraic semantics of Font's calculus, whereas $\B$ has no equivalent algebraic semantics. A cosmetic difference between these calculi and the ones presented here is that we consider the truth constants $\top$ and $\bot$ to be part of the signature of $\B$, while Pynko and Font do~not.} On the other hand, we may also draw a connection between the consequence relation of $\B$ and the derivability relation between sequents and sets of sequents, as Pynko \cite{pynko10} does. In the simplest form, this correspondence says that $\Gamma \vdash_{\B} \varphi$ if and only if the sequent $\emptyset \seq \varphi$ is provable from the sequents $\emptyset \seq \gamma$ for $\gamma \in \Gamma$. It is well known that these two relations coincide e.g.\ in some standard Gentzen calculi for classical and intuitionistic logic, but when it comes to $\B$, the two approaches require us to adopt rather different Gentzen calculi. In particular, the calculi of Pynko \cite{pynko95} and Font~\cite{font97} contain the Identity axiom but neither the standard introduction rules for negation nor the elimination rules, while the calculus of Pynko \cite{pynko10} it contains both the standard introduction rules for negation and the elimination rules while leaving out Identity and Cut. In the present paper we opt for the latter approach.


  The paper is structured as follows. Section \ref{sec:preliminaries} contains all the necessary preliminaries regarding super-Belnap logics and the equivalence of Gentzen and Hilbert relations. Section \ref{sec:gentzen-calculi} introduces Gentzen calculi for super-Belnap logics and Section \ref{sec:analytic--synthetic} generalizes the cut elimination theorem for classical logic to these calculi. In particular, the cut elimination theorem is generalized to proofs from non-empty sets of sequential premises: each such proof can be transformed into what we call a structurally atomic analytic--synthetic proof. Section \ref{sec:interpolation} then exploits this generalization of the cut elimination theorem to prove new interpolation results for super-Belnap logics. In particular, we obtain an alternative Gentzen-style proof of Milne's recent refinement of the Craig interpolation theorem for classical logic \cite{milne16}.

  Starting from the observation due to Pynko that a Gentzen calculus for the Dunn--Belnap logic $\B$ may be obtained by adding the inverses of the logical introduction rules to a standard Gentzen calculus for classical logic and dropping Identity and Cut (Theorem \ref{thm:hilbertizability}), we show that each super-Belnap logic may be obtained by adding a set of structural rules to this calculus (Proposition \ref{prop:structural-axiomatization}), in particular adding Identity yields a calculus for $\LP$ and adding Cut yields a calculus for $\K$, while adding a limited form of Cut yields a calculus for the Exactly True Logic of Pietz and Rivieccio~\cite{pietz+rivieccio13}. The fact that the logics $\CL$ and $\LP$ have the same theorems may be used to prove the classical admissibility of Cut (Theorem \ref{thm:cut-admissible}), as remarked on by Pynko already. Dually, the fact that the logics $\CL$ and $\K$ have the same antitheorems may be used to prove the classical anti\-admissibility of Identity (Theorem \ref{thm:identity-antiadmissible}). We then introduce structurally atomic analytic--synthetic proofs as a generalization of cut-free proofs to proofs from non-empty sets of sequential premises, and provide a procedure for transforming arbitrary proofs in suitable super-Belnap calculi into such proofs (Proposition~\ref{prop:structurally-atomic-specific} and Theorem \ref{thm:normal-form}). This is exploited in the last section to establish a broad sufficient condition for (a stronger form of) interpolation, proving new inter\-polation results for super-Belnap logics and strengthening some old ones for $\LP$~and~$\K$ (Propositions \ref{prop:interpolation}--\ref{prop:cl-interpolation}).

  We emphasize that we only consider propositional logics in this paper.

\section{Preliminaries}
\label{sec:preliminaries}

  In this section, we briefly introduce several super-Belnap logics and some definitions concerning the equivalence of Hilbert and Gentzen relations. For more details about super-Belnap logics and their history and motivations, the reader may consult the papers \cite{rivieccio12,albuquerque+prenosil+rivieccio16,prenosil16}.

  We shall be using the paradigm of abstract algebraic logic, see e.g.\ Font's recent textbook \cite{font16}. In particular, by a \emph{logic}, also sometimes called a \emph{Hilbert relation}, we shall mean a relation $\logic{L}$ between sets of formulas and formulas in a given language, written $\Gamma \vdash_{\logic{L}} \varphi$, which satisfies the following conditions:
\begin{align*}
& \varphi \vdash \varphi \tag{reflexivity}\\
& \text{if } \Gamma \vdash \varphi \text{, then } \Gamma, \Delta \vdash \varphi \tag{monotonicity} \\
& \text{if } \Gamma \vdash \varphi \text{ for all } \varphi \in \Phi \text{ and } \Phi, \Delta \vdash \psi \text{, then } \Gamma, \Delta \vdash \psi  \tag{cut} \\
& \text{if }\Gamma \vdash \varphi \text{, then } \sigma[\Gamma] \vdash \sigma \varphi \text{ for each substitution } \sigma \tag{structurality}
\end{align*}
If moreover $\Gamma \vdash \varphi$ implies that there is some finite set of formulas $\Gamma' \subseteq \Gamma$ such that $\Gamma' \vdash \varphi$, then the logic is called \emph{finitary}.

  A \emph{(logical) matrix} is an algebra $\algebra{A}$ in a given language equipped with a set of designated values $F \subseteq \algebra{A}$. Each matrix $\langle \algebra{A}, F \rangle$ defines a logic $\Log \langle \algebra{A}, F \rangle$ such that
\begin{align*}
  \Gamma \vdash_{\Log \langle \algebra{A}, F \rangle} \varphi \text{ if and only if } v[\Gamma] \subseteq F \text{ implies } v(\varphi) \in F
\end{align*}
for each valuation $v: \FmAlgebra \rightarrow \algebra{A}$, where $\FmAlgebra$ denotes the algebra of formulas in the given language. If $\class{K}$ is a class of matrices, then we define $\Log \class{K} = \bigcap_{\algebra{A} \in \class{K}} \Log \algebra{A}$. Each logic is then the logic of some class of matrices.

  Suppose that a logic $\logic{L}$ has a constant $\bot$ such that $\bot \vdash_{\logic{L}} p$ for some variable $p$. Then $\bot$ is never designated in a non-trivial model of $\logic{L}$. By an \emph{explosive rule} in such a logic, we shall now mean a rule of the form $\Gamma \vdash \bot$, denoted also $\Gamma \vdash \emptyset$. An \emph{explosive extension} of $\logic{L}$ is then an extension of $\logic{L}_{exp}$ by a set of explosive rules. If $\logic{L}_{exp}$ is an explosive extension of $\logic{L}$, then $\Gamma \vdash_{\logic{L}_{exp}} \varphi$ if and only if either $\Gamma \vdash_{\logic{L}} \varphi$ or $\Gamma \vdash_{\logic{L}_{exp}} \emptyset$.

  \emph{Theorems} of a logic $\logic{L}$ are formulas $\varphi$ such that $\emptyset \vdash_{\logic{L}} \varphi$. \emph{Antitheorems} of~$\logic{L}$ are sets of formulas $\Gamma$ such that $\Gamma \vdash_{\logic{L}} \emptyset$.

\begin{definition} \label{def:admissible}
  An inference rule is \emph{admissible} (\emph{antiadmissible}) in a logic if adding it does not yield any new theorems (antitheorems).
\end{definition}

  We now recall some (slightly adapted) definitions introduced by Raftery in his paper \cite{raftery06} concerning the equivalence of Hilbert and Gentzen systems. Raftery takes sequents to be pairs of finite sequences, but for the sake of simplicity we shall adopt the definition that a \emph{sequent} is a pair of finite \emph{multisets} of formulas, written as $\Gamma \seq \Delta$. This obviates the need for explicitly introducing the structural rule of Exchange. An \emph{atomic sequent} is a sequent in which all formulas are atoms. The \emph{empty sequent} is the sequent $\emptyset \seq \emptyset$.

  A \emph{Gentzen relation} $\GentzenRelation \logic{L}$ is a relation between sets of sequents and sequents which satisfies natural analogues of the above four conditions defining Hilbert relations. A Gentzen relation is \emph{finitary} if $\SequentSet \vdash_{\GentzenRelation \logic{L}} \Gamma \seq \Delta$ implies that there is some finite $\SequentSet' \subseteq \SequentSet$ such that $\SequentSet' \vdash_{\GentzenRelation \logic{L}} \Gamma \seq \Delta$.

  We say that two sets of formulas $\Gamma$ and $\Delta$ are equivalent in $\logic{L}$ in case $\Gamma \vdash_{\logic{L}} \delta$ for each $\delta \in \Delta$ and $\Delta \vdash_{\logic{L}} \gamma$ for each $\gamma \in \Gamma$, and likewise for sequents. Two sets of rules are equivalent in a Hilbert or Gentzen relation if the least Hilbert or Gentzen relations extending the relation by those rules coincide.

  A Hilbert or Gentzen calculus is just a set of rules in a given language, i.e.\ pairs of sets of formulas (sequents) interpreted as premises and formulas (sequents) interpreted as conclusions. A calculus axiomatizes a Hilbert or Gentzen relation if it is the least such relation which contains all of the rules of the calculus. Proofs in a given calculus are defined as suitably labelled well-founded trees (trees where all branches are finite), and a formula or a sequent is provable in a calculus from a set of formulas or sequents if and only if it the corresponding rule is valid in the relation axiomatized by the calculus. In finitary calculi, proofs are suitably labelled finite trees.

  A Gentzen relation $\GentzenRelation \logic{L}$ and a Hilbert relation $\logic{L}$ are \emph{simply equivalent} if there is a definable transformer $\tautrans$ from sequents to sets of formulas in the sense of Raftery \cite{raftery06} (adapted to multisets) such that
\begin{align*}
  & \SequentSet \vdash_{\GentzenRelation \logic{L}} \Gamma \seq \Delta \text{ if and only if } \tautrans [\SequentSet] \vdash_{\logic{L}} \tautrans(\Gamma \seq \Delta), \\
  & \Gamma \vdash_{\logic{L}} \varphi \text{ if and only if } \rhotrans [\SequentSet] \vdash_{\logic{L}} \rhotrans(\varphi), \\
  & \Gamma \seq \Delta \dashv\vdash_{\GentzenRelation \logic{L}} \rhotrans \tautrans (\Gamma \seq \Delta), \\
  & \varphi \dashv\vdash_{\logic{L}} \tautrans \rhotrans (\varphi),
\end{align*}
where $\rhotrans$ is the inclusion transformer $\varphi \mapsto \{ \emptyset \seq \varphi \}$. Note that the second and third conditions (or alternatively, the first and fourth conditions) are in fact redundant in this definition.

  In the following, we will only be concerned with Gentzen relations which validate the rule of Weakening, hence $\SequentSet \vdash_{\GentzenRelation \logic{L}} \emptyset \seq \emptyset$ will imply $\SequentSet \vdash_{\GentzenRelation \logic{L}} \Gamma \seq \Delta$ for each sequent $\Gamma \seq \Delta$. For such Gentzen relations, we may therefore define $\SequentSet \vdash_{\GentzenRelation \logic{L}} \emptyset$ as an abbreviation for $\SequentSet \vdash_{\GentzenRelation \logic{L}} \emptyset \seq \emptyset$. The antitheorems and antiadmissible rules of a Gentzen relation which validates the rule of Weakening may then be defined as in the case of Hilbert relations, replacing the constant $\bot$ by the empty sequent $\emptyset \seq \emptyset$.

  We now turn our attention to \emph{super-Belnap logics}, i.e.\ extensions of $\B$. The signature of the logic $\B$ consists of the distributive lattice conjunction and disjunction connectives $\wedge$ and $\vee$, the De Morgan negation operator $\dmneg$, and the two truth constants $\top$ and $\bot$ representing the top and bottom of the lattice order, respectively. Figure \ref{fig:matrices}, borrowed from the paper \cite{prenosil16}, shows several important matrices which define the super-Belnap logics studied in this paper. In all of these matrices, the De Morgan negation corresponds to reflection across the horizontal axis of symmetry, while the other connectives and constants are interpreted according to the lattice order.

\begin{figure}
\caption{Some logical matrices for super-Belnap logics}
\label{fig:matrices}

\bigskip

\begin{center}
\begin{tikzpicture}[scale=1,
  dot/.style={circle,fill,inner sep=2.5pt,outer sep=2.5pt}]
  \node (DM4a) at (0,-1) [dot] {};
  \node (DM4b) at (-1,0) [dot] {};
  \node (DM4c) at (1,0) [dot] {};
  \node (DM4d) at (0,1) [dot] {};
  \draw[-] (DM4a) edge (DM4b);
  \draw[-] (DM4a) edge (DM4c);
  \draw[-] (DM4b) edge (DM4d);
  \draw[-] (DM4c) edge (DM4d);
  \draw[rotate around={45:(0.5,0.5)}] (0.5,0.5) ellipse (0.5 and 1);
  \node at (0,-2) {$\algebra{B_{4}}$};
\end{tikzpicture}
\qquad
\begin{tikzpicture}[scale=1,
  dot/.style={circle,fill,inner sep=2.5pt,outer sep=2.5pt}]
  \node (K3a) at (0,0) [dot] {};
  \node (K3b) at (0,1) [dot] {};
  \node (K3c) at (0,2) [dot] {};
  \draw[-] (K3a) edge (K3b);
  \draw[-] (K3b) edge (K3c);
  \draw (0,2) ellipse (0.5 and 0.5);
  \node at (0,-1) {$\algebra{K_{3}}$};
\end{tikzpicture}
\qquad
\begin{tikzpicture}[scale=1,
  dot/.style={circle,fill,inner sep=2.5pt,outer sep=2.5pt}]
  \node (K3a) at (0,0) [dot] {};
  \node (K3b) at (0,1) [dot] {};
  \node (K3c) at (0,2) [dot] {};
  \draw[-] (K3a) edge (K3b);
  \draw[-] (K3b) edge (K3c);
  \draw (0,1.5) ellipse (0.5 and 1);
  \node at (0,-1) {$\algebra{LP_{3}}$};
\end{tikzpicture}
\qquad
\begin{tikzpicture}[scale=1,
  dot/.style={circle,fill,inner sep=2.5pt,outer sep=2.5pt}]
  \node (DM4a) at (0,-1) [dot] {};
  \node (DM4b) at (-1,0) [dot] {};
  \node (DM4c) at (1,0) [dot] {};
  \node (DM4d) at (0,1) [dot] {};
  \draw[-] (DM4a) edge (DM4b);
  \draw[-] (DM4a) edge (DM4c);
  \draw[-] (DM4b) edge (DM4d);
  \draw[-] (DM4c) edge (DM4d);
  \draw (0,1) ellipse (0.5 and 0.5);
  \node at (0,-2) {$\algebra{ETL_{4}}$};
\end{tikzpicture}
\end{center}

\end{figure}
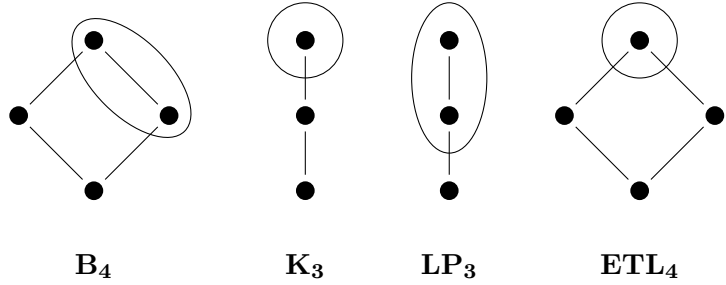

  The most important super-Belnap logics may now be introduced semantically as follows: $\B = \Log \Bmatrix$, $\LP = \Log \LPmatrix$, $\K = \Log \Kmatrix$, and $\ETL = \Log \ETLmatrix$. That is, $\ETL$ is the Exactly True Logic of Pietz and Rivieccio \cite{pietz+rivieccio13}. The matrix $\Bmatrix$ comes with a natural \emph{information order}, denoted $\infleq$, which then specializes to the two submatrices $\LPmatrix$ and $\Kmatrix$ of $\Bmatrix$. This order is obtained by reading the diagram of $\Bmatrix$ left-to-right rather than bottom-up. It extends to valuations $v, w: \FmAlgebra \rightarrow \Bmatrix$ in the natural way: $v \infleq w$ if and only if $v(p) \infleq w(p)$ for each atom $p$. A valuation $v: \FmAlgebra \rightarrow \Bmatrix$ is called \emph{classical} if its range is the two-element Boolean submatrix of $\Bmatrix$.

\begin{proposition} \label{prop:information monotone}
  The connectives of $\Bmatrix$, $\LPmatrix$, and $\Kmatrix$ are monotone with respect to the information order.
\end{proposition}

  The information order on $\Bmatrix$, $\LPmatrix$, and $\Kmatrix$ is precisely the Leibniz order in the sense of Raftery \cite{raftery13} if we assign all polynomials to the positive polarity. By contrast, the Leibniz order on $\ETLmatrix$ is the identity relation.

  We now axiomatize the super-Belnap logics introduced above relative $\B$: the logic $\K$ is axiomatized by the rule of resolution $p \vee q, \dmneg q \vee r \vdash p \vee r$, the logic $\ETL$ is axiomatized by the rule of disjunctive syllogism $p, \dmneg p \vee q \vdash q$, the logic $\LP$ is axiomatized by the law of the excluded middle $\emptyset \vdash p \vee \dmneg p$. The logic $\Kleq = \K \cap \LP$, called Kleene's logic of order by Font \cite{font97}, is axiomatized by the rule $(p \wedge \dmneg p) \vee r \vdash (q \vee \dmneg q) \vee r$, as proved in \cite{albuquerque+prenosil+rivieccio16}. Moreover, we define $\ECQ$ as the extension of $\B$ by the rule $p, \dmneg p \vdash \emptyset$. It was shown in \cite{prenosil16} that $\ECQ$ is the logic of the matrix $\ETLmatrix \times \Bmatrix$.

  The join of two logics is defined as the least logic extending both of them. In particular, it is known that $\CL = \LP \vee \K = \LP \vee \ETL$.

\section{Gentzen calculi for super-Belnap logics}
\label{sec:gentzen-calculi}

  We now introduce a natural Gentzen calculus which is simply equivalent to the Dunn--Belnap logic $\B$ in the sense of Raftery and describe some of its extensions. We then consider the interderivability, admissibility, and antiadmissibility of the introduction and elimination rules for the logical connectives in the Gentzen calculi for $\CL$, $\LP$, and $\K$, and use these facts to provide a semantic proof of the admissibility of Cut and the antiadmissibility of Identity in the standard Gentzen calculus for classical logic.

  The Gentzen calculus $\GKforB$ defined by the rules in Figure \ref{fig:b-calculus} may be described simply as a standard Gentzen calculus for classical logic without the rules of Cut and Identity but with the inverse of each introduction rule for each logical connective or constant. For example, the notation
\begin{prooftree}
\def\fCenter{\seq}
\Axiom$\Gamma \fCenter \Delta, \varphi$
\Axiom$\Gamma \fCenter \Delta, \psi$
\doubleLine
\BinaryInf$\Gamma \fCenter \Delta, \varphi \wedge \psi$
\end{prooftree}
is meant to indicate that the calculus contains the three rules
\begin{prooftree}
\def\fCenter{\seq}
\Axiom$\Gamma \fCenter \Delta, \varphi$
\Axiom$\Gamma \fCenter \Delta, \psi$
\BinaryInf$\Gamma \fCenter \Delta, \varphi \wedge \psi$

\def\fCenter{\seq}
\Axiom$\Gamma \fCenter \Delta, \varphi \wedge \psi$
\UnaryInf$\Gamma \fCenter \Delta, \varphi$

\def\fCenter{\seq}
\Axiom$\Gamma \fCenter \Delta, \varphi \wedge \psi$
\UnaryInf$\Gamma \fCenter \Delta, \psi$

\noLine
\TrinaryInfC{}
\end{prooftree}
the first one being called the right introduction rule for conjunction and the second and third being called the right elimination rules for conjunction. The sequents $\emptyset \seq \top$ and $\bot \seq \emptyset$ are the only two axioms of this Gentzen calculus. The Gentzen relation axiomatized by $\GKforB$ will be called $\GB$.

\begin{figure}[t]
\caption{The Gentzen calculus $\GKforB$ for $\B$}
\label{fig:b-calculus}

\medskip

\textbf{Logical rules}

\begin{prooftree}
\def\fCenter{\seq}
\Axiom$\Gamma \fCenter \Delta, \varphi$
\Axiom$\Gamma \fCenter \Delta, \psi$
\doubleLine
\BinaryInf$\Gamma \fCenter \Delta, \varphi \wedge \psi$

\def\fCenter{\seq}
\Axiom$\varphi, \psi, \Gamma \fCenter \Delta$
\doubleLine
\UnaryInf$\varphi \wedge \psi, \Gamma \fCenter \Delta$

\noLine
\BinaryInfC{}
\end{prooftree}

\begin{prooftree}
\def\fCenter{\seq}
\Axiom$\varphi, \Gamma \fCenter \Delta$
\Axiom$\psi, \Gamma \fCenter \Delta$
\doubleLine
\BinaryInf$\varphi \vee \psi, \Gamma \fCenter \Delta$

\def\fCenter{\seq}
\Axiom$\Gamma \fCenter \Delta, \varphi, \psi$
\doubleLine
\UnaryInf$\Gamma \fCenter \Delta, \varphi \vee \psi$

\noLine
\BinaryInfC{}
\end{prooftree}

\begin{prooftree}
\def\fCenter{\seq}
\Axiom$\varphi, \Gamma \fCenter \Delta$
\doubleLine
\UnaryInf$\Gamma \fCenter \Delta, \dmneg \varphi$

\def\fCenter{\seq}
\Axiom$\Gamma \fCenter \Delta, \varphi$
\doubleLine
\UnaryInf$\dmneg \varphi, \Gamma \fCenter \Delta$

\noLine
\BinaryInfC{}
\end{prooftree}

\begin{prooftree}
\def\fCenter{\seq}

\Axiom$\emptyset \fCenter \top$

\Axiom$\top, \Gamma \fCenter \Delta$
\UnaryInf$\Gamma \fCenter \Delta$

\noLine
\BinaryInfC{}

\def\fCenter{\seq}
\Axiom$\Gamma \fCenter \Delta, \bot$
\UnaryInf$\Gamma \fCenter \Delta$

\Axiom$\bot \fCenter \emptyset$

\noLine
\BinaryInfC{}

\noLine
\BinaryInfC{}
\end{prooftree}

\textbf{Structural rules}

\begin{prooftree}
\def\fCenter{\seq}
\Axiom$\Gamma \fCenter \Delta$
\UnaryInf$\varphi, \Gamma \fCenter \Delta$

\def\fCenter{\seq}
\Axiom$\Gamma \fCenter \Delta$
\UnaryInf$\Gamma \fCenter \Delta, \varphi$

\noLine
\BinaryInfC{}

\def\fCenter{\seq}
\Axiom$\varphi, \varphi, \Gamma \fCenter \Delta$
\UnaryInf$\varphi, \Gamma \fCenter \Delta$

\def\fCenter{\seq}
\Axiom$\Gamma \fCenter \Delta, \varphi, \varphi$
\UnaryInf$\Gamma \fCenter \Delta, \varphi$

\noLine
\BinaryInfC{}

\noLine
\BinaryInfC{}
\end{prooftree}

\end{figure}

\begin{theorem}[\cite{pynko10}] \label{thm:hilbertizability}
  The Gentzen relation $\GB$ is simply equivalent to the Hilbert relation $\B$ via the transformer $\tautrans: \Gamma \seq \Delta \mapsto \{ \dmneg \bigwedge \Gamma \vee \bigvee \Delta \}$.
\end{theorem}

  The transformers $\tautrans: \Gamma \seq \Delta \mapsto \{ \dmneg \bigwedge \Gamma \vee \bigvee \Delta \}$ and $\rhotrans: \varphi \mapsto \{ \emptyset \seq \varphi \}$ therefore define two mutually inverse isomorphisms between the lattice of Hilbert relations extending $\B$ and the lattice of Gentzen relations extending $\GB$, and moreover these isomorphisms preserve finitarity \cite[Corollary 7.5]{raftery06}. In other words, when we talk about (finitary) super-Belnap logics, we may as well be talking about (finitary) Gentzen relations extending $\GB$.

  Moreover, it takes but a moment's reflection to see that each extension of $\GB$ may in fact be axiomatized by \emph{structural rules}, i.e.\ rules which do not contain any occurrence of a logical connective or constant. Note that what we call Cut, Weakening, and Contraction are strictly speaking sets of structural rules. For example, Contraction is a set of rules which includes the rules $p, p \seq q \vdash p \seq q$ and $p, q, q \seq r \vdash p, q \seq r$ etc.

\begin{proposition} \label{prop:sequent-composition-decomposition}
  Each sequent is equivalent in the Gentzen relation $\GB$ to a finite set of atomic sequents. Each finite set of sequents is equivalent in the Gentzen relation $\GB$ to a single sequent of the form $\emptyset \seq \varphi$.
\end{proposition}

\begin{proof}
  Both of these claims may be proved by a straightforward induction over the complexity of the sequent or the finite set of sequents.
\end{proof}

\begin{proposition} \label{prop:structural-axiomatization}
  Each (finitary) extension of $\GB$ may be axiomatized as an extension of the calculus $\GKforB$ by a set of (finitary) structural rules.
\end{proposition}

\begin{proof}
  It suffices to show that each (finitary) Gentzen rule is equivalent in $\GB$ to a (finitary) structural rule. Let $\set{\Gamma_{i} \seq \Delta_{i}}{i \in I} \vdash \Gamma \seq \Delta$ be a Gentzen rule. By Proposition \ref{prop:sequent-composition-decomposition} there is for each $i \in I$ a finite set of atomic sequents $\SequentSet_{i}$ equivalent to $\Gamma_{i} \seq \Delta_{i}$ and there is a finite set of atomic sequents $\SequentSet$ equivalent to $\Gamma \seq \Delta$. Then the Gentzen rule in question is equivalent to the finite set of rules $\bigcup_{i \in I} \SequentSet_{i} \vdash \Lambda \seq \Pi$ for $\Lambda \seq \Pi \in \SequentSet$. Moreover, these rules are finitary if the Gentzen rule in question is finitary.
\end{proof}

  If a Hilbert relation $\logic{L}$ is the extension of $\B$ by the rules $\Gamma_{i} \vdash \varphi_{i}$ for $i \in I$, then the rules $\rhotrans[\Gamma_{i}] \vdash \rhotrans(\varphi_{i})$ axiomatize the corresponding Gentzen relation $\GentzenRelation \logic{L}$ extending $\GentzenRelation \logic{B}$ by \cite[Theorem 7.1]{raftery06}. Proposition \ref{prop:structural-axiomatization} now claims that we may always simplify such an axiomatization to a structural axiomatization. For example, the rule $p \wedge \dmneg p \vdash q \vee \dmneg q$ axiomatizing the logic $\LP \cap \ECQ$ corresponds to the Gentzen rule (schema)
\begin{prooftree}
\def\fCenter{\seq}
\Axiom$\emptyset \fCenter \varphi$
\Axiom$\varphi \fCenter \emptyset$
\BinaryInf$\psi \fCenter \psi$
\end{prooftree}
while the rule $(p \wedge \dmneg p) \vee r \vdash (q \vee \dmneg q) \vee r$ axiomatizing the logic $\Kleq$ corresponds to the Gentzen rule (schema)
\begin{prooftree}
\def\fCenter{\seq}
\Axiom$\Gamma \fCenter \Delta, \varphi$
\Axiom$\varphi, \Gamma \fCenter \Delta$
\BinaryInf$\psi, \Gamma \fCenter \Delta, \psi$
\end{prooftree}
which may be seen as a combination of Identity and Cut.

  Extensions of the Gentzen calculus $\GKforB$ by a set of structural rules will be called \emph{super-Belnap (Gentzen) calculi}. Our goal here is to develop the rudiments of the proof theory of super-Belnap calculi. The rules which each super-Belnap calculus has in addition to the rules of $\GKforB$ will be called its \emph{specific structural rules}, as opposed to the \emph{common structural rules} of Weakening and Contraction shared by all super-Belnap calculi.

  Note that although the introduction and elimination rules of super-Belnap calculi are finitary, we do not exclude non-finitary structural rules. (The existence of non-finitary super-Belnap logics was proved in \cite{prenosil16}.) Proofs in such calculi are well-founded trees, i.e.\ trees with no infinite branches.

\begin{figure}[t]
\caption{Some specific structural rules}
\label{fig:specific-structural-rules}

\medskip

\textbf{Identity}

\begin{prooftree}
\def\fCenter{\seq}
\Axiom$\varphi \fCenter \varphi$
\end{prooftree}

\textbf{Cut}

\begin{prooftree}
\def\fCenter{\seq}
\Axiom$\Gamma \fCenter \Delta, \varphi$
\Axiom$\varphi, \Gamma' \fCenter \Delta$
\BinaryInf$\Gamma, \Gamma' \fCenter \Delta, \Delta'$
\end{prooftree}

\textbf{Limited Cut}

\begin{prooftree}
\def\fCenter{\seq}
\Axiom$\emptyset \fCenter \varphi$
\Axiom$\varphi, \Gamma \fCenter \Delta$
\BinaryInf$\Gamma \fCenter \Delta$

\Axiom$\Gamma \fCenter \Delta, \varphi$
\Axiom$\varphi \fCenter \emptyset$
\BinaryInf$\Gamma \fCenter \Delta$

\noLine
\BinaryInfC{}
\end{prooftree}

\textbf{Explosive Cut}

\begin{prooftree}
\def\fCenter{\seq}
\Axiom$\emptyset \fCenter \varphi$
\Axiom$\varphi \fCenter \emptyset$
\BinaryInf$\emptyset \fCenter \emptyset$
\end{prooftree}

\end{figure}

  Some of the important specific structural rules are listed in Figure \ref{fig:specific-structural-rules}. Identity corresponds to the rule $\emptyset \vdash p \vee \dmneg p$ which axiomatizes $\LP$, while Cut corresponds to the rule $p \vee \dmneg q, q \vee r \vdash p \vee r$ which axiomatizes $\K$. Limited Cut in either of the two equivalent forms corresponds to the rule $p, \dmneg p \vee q \vdash q$ which axiomatizes $\ETL$, and Explosive Cut corresponds to the rule $p, \dmneg p \vdash \emptyset$ which axiomatizes $\ECQ$. Adding these rules to $\GKforB$ yields the calculi $\GKforLP$, $\GKforK$, $\GKforETL$, and $\GKforECQ$ which axiomatize the Gentzen versions of these logics.

\begin{proposition} \label{prop:cut-from-limited-cut}
  Cut is a derivable rule in each super-Belnap calculus which contains both Identity and Limited Cut.
\end{proposition}

\begin{proof}
  Recall that $\CL = \LP \vee \K = \LP \vee \ETL$.
\end{proof}

  We now turn our attention to the super-Belnap calculi for the best known extensions of $\B$, namely $\CL$, $\LP$, and $\K$. It turns out that for $\CL$ we may pick and choose any combination of introduction and elimination rules, provided we include at least one of these for each connective. For $\LP$ we may do without the elimination rules provided that we are only interested in which sequents are derivable from the empty set of sequents, and dually for $\K$ we may do without the introduction rules provided that we are only interested in sets of sequents from which the empty sequent is derivable.

\begin{proposition} \label{prop:intro-and-elim-equivalent}
   In the presence of Identity and Cut, the left (right) intro\-duction and elimination rules for each connective are interderivable.
\end{proposition}

\begin{proof}
  We only prove the claim for conjunction, as the argument for dis\-junction is dual and the arguments for negation and for the truth constants are simpler. To simulate the left intro\-duction rule by the left elimination rule, we use the following strategy:
\begin{prooftree}
\def\fCenter{\seq}

\Axiom$\varphi \wedge \psi \fCenter \varphi \wedge \psi$
\UnaryInf$\varphi \wedge \psi \fCenter \psi$

\Axiom$\varphi \wedge \psi \fCenter \varphi \wedge \psi$
\UnaryInf$\varphi \wedge \psi \fCenter \varphi$

\Axiom$\varphi, \psi, \Gamma \fCenter \Delta$
\BinaryInf$\varphi \wedge \psi, \psi, \Gamma \fCenter \Delta$

\BinaryInf$\varphi \wedge \psi, \varphi \wedge \psi, \Gamma \fCenter \Delta$
\UnaryInf$\varphi \wedge \psi, \Gamma \fCenter \Delta$
\end{prooftree}
\medskip
The left elimination rule may then be simulated by the following proof:
\begin{prooftree}
\def\fCenter{\seq}

\Axiom$\varphi \fCenter \varphi$
\UnaryInf$\varphi, \psi \fCenter \varphi$
\Axiom$\psi \fCenter \psi$
\UnaryInf$\varphi, \psi \fCenter \psi$
\BinaryInf$\varphi, \psi \fCenter \varphi \wedge \psi$
\Axiom$\varphi \wedge \psi, \Gamma \fCenter \Delta$
\BinaryInf$\varphi, \psi, \Gamma \fCenter \Delta$
\end{prooftree}
To simulate the right introduction rule for conjunction by the left elimination rule, we use the following strategy:
\begin{prooftree}
\def\fCenter{\seq}

\Axiom$\Gamma \fCenter \Delta, \psi$

\Axiom$\Gamma \fCenter \Delta, \varphi$

\Axiom$\varphi \wedge \psi \fCenter \varphi \wedge \psi$
\UnaryInf$\varphi, \psi \fCenter \varphi \wedge \psi$

\BinaryInf$\psi, \Gamma \fCenter \Delta, \varphi \wedge \psi$
\BinaryInf$\Gamma, \Gamma \fCenter \Delta, \Delta, \varphi \wedge \psi$
\UnaryInf$\Gamma \fCenter \Delta, \varphi \wedge \psi$
\end{prooftree}
\medskip
The right elimination rules may then be simulated by the following proofs:
\begin{prooftree}
\def\fCenter{\seq}

\Axiom$\Gamma \fCenter \Delta, \varphi \wedge \psi$
\Axiom$\varphi \fCenter \varphi$
\UnaryInf$\varphi, \psi \fCenter \varphi$
\UnaryInf$\varphi \wedge \psi \fCenter \varphi$
\BinaryInf$\Gamma \fCenter \Delta, \varphi$

\Axiom$\Gamma \fCenter \Delta, \varphi \wedge \psi$
\Axiom$\psi \fCenter \psi$
\UnaryInf$\varphi, \psi \fCenter \psi$
\UnaryInf$\varphi \wedge \psi \fCenter \psi$
\BinaryInf$\Gamma \fCenter \Delta, \psi$

\noLine
\BinaryInfC{}
\end{prooftree}

\end{proof}

  We need both Identity and Cut to establish the interderivability of the intro\-duction and elimination rules. However, if we are only interested in the admissibility (antiadmissibility) of the elimination (introduction) rules, the presence of Cut (Identity) is not needed. This claim (for admissibility) is called the Inversion Lemma by Troelstra and Schwichtenberg \cite{basicprooftheory}, and it constitutes a step in the standard proof of cut elimination for classical logic.

\begin{proposition}[\cite{basicprooftheory}, Proposition 3.5.4, p.~79] \label{prop:elim-admissible}
  Each elimination rule is admissible in the calculus which contains the common structural rules, the introduction rules for all connectives, and Identity.
\end{proposition}

  This proposition explains why the elimination rules are invisible from the standard point of view, which is concerned only with sequents provable from an empty set of premises.

\begin{proposition} \label{prop:intro-admissible}
  Each introduction rule is antiadmissible in the calculus which contains the common structural rules, the elimination rules for all connectives, and Cut.
\end{proposition}

\begin{proof}
  Let us say that a sequent $\Gamma \seq \Delta$ \emph{explodes directly} relative to a multiset of sequents $\SequentSet$ if there is a proof of $\emptyset \seq \emptyset$ from $S \cup \{ \Gamma \seq \Delta \}$ which does not contain any introduction rules and in which each sequent is used precisely the specified number of times as a premise of the proof. Explicit reference to the set of side assumptions $S$ will be suppressed in the following. It now suffices to prove by induction over the height $h$ of the premise $\Gamma \seq \Delta$ in the proof of $\emptyset \seq \emptyset$ that
\begin{itemize}
\item if $\varphi \wedge \psi, \Gamma \seq \Delta$ explodes directly, so does $\varphi, \psi, \Gamma \seq \Delta$
\item if $\Gamma \seq \Delta, \varphi \wedge \psi$ explodes directly, so does $\{ \Gamma \seq \Delta, \varphi \} \cup \{ \Gamma \seq \Delta, \psi \}$
\item if $\varphi \vee \psi, \Gamma \seq \Delta$ explodes directly, so does $\{ \varphi, \Gamma \seq \Delta \} \cup \{ \varphi, \Gamma \seq \Delta \}$
\item if $\Gamma \seq \Delta, \varphi \vee \psi$ explodes directly, so does $\Gamma \seq \Delta, \varphi, \psi$
\item if $\dmneg \varphi, \Gamma \seq \Delta$ explodes directly, so does $\Gamma \seq \Delta, \varphi$
\item if $\Gamma \seq \Delta, \dmneg \varphi$ explodes directly, so does $\varphi, \Gamma \seq \Delta$
\item if $\top, \Gamma \seq \Delta$ explodes directly, so does $\Gamma \seq \Delta$
\item if $\emptyset \seq \top$ explodes directly, so does $\emptyset$
\item if $\bot \seq \emptyset$ explodes directly, so does $\emptyset$
\item if $\Gamma \seq \Delta, \bot$ explodes directly, so does $\Gamma \seq \Delta$
\end{itemize}
  We only deal with the first item, the rest of them are entirely analogical. The sequent $\varphi \wedge \psi, \Gamma \seq \Delta$ cannot be the last sequent of a proof of $\emptyset \seq \emptyset$, hence the base case holds trivially. Now suppose that an instance of $\varphi \wedge \psi, \Gamma \seq \Delta$ explodes directly via a proof where this sequent has height $h + 1$. If the rule which follows this instance of $\varphi \wedge \psi, \Gamma \seq \Delta$ is any rule other than an elimination rule applied to $\varphi \wedge \psi$, we may simply use the inductive hypothesis for~$h$. If, on the other hand, this instance of $\varphi \wedge \psi, \Gamma \seq \Delta$ occurs as a premise of the left conjunction elimination rule, then clearly $\varphi, \psi, \Gamma \seq \Delta$ explodes directly (with respect to the same multiset of sequents $\SequentSet$).
\end{proof}

  The above proof is height-preserving in the same sense as the original Inversion Lemma. Notice the slight twist involving multisets forced on us by the fact that a proof has only one conclusion but it may have many premises.

  The resolution calculus for classical logic is essentially the restriction of the calculus $\GKforK$ to atomic sequents. Proposition \ref{prop:intro-admissible} is there\-fore a more sophisticated version of the trivial observation that we never need to apply resolution to clauses of the form $p \vee \dmneg p$.

  Having proved the Inversion Lemma and its dual, we may now prove the admissibility of Cut and the antiadmissibility of Identity in the standard calculus for classical logic using a semantic argument. This is because for us these are not mere fragments of some calculus, but rather calculi in their own right with a perfectly good semantics provided by the logics $\LP$ and $\K$. (By contrast, the Inversion Lemma is not needed to prove the admissibility of Cut in the non-deterministic framework of Lahav and Avron \cite{lahav+avron13}.)

  The following proposition is already well-known (at least its first part) but we include a proof nonetheless for the sake of being self-contained and showing that the argument is indeed purely semantic.

\begin{proposition}[Weak conservativity] \label{prop:weak-conservativity}
  The logics $\CL$ and $\LP$ have the same theorems. The logics $\CL$ and $\K$ have the same antitheorems.
\end{proposition}

\begin{proof}
  Each theorem of $\LP$ is a theorem of $\CL$. Conversely, suppose that $v: \FmAlgebra \rightarrow \LPmatrix$ is a valuation such that $v(\varphi)$ is not designated. Then in particular $v(\varphi)$ is a minimal element of the information order on $\LPmatrix$. Since the operations of $\LPmatrix$ are monotone with respect to the information order by Proposition \ref{prop:information monotone}, it follows that $w(\varphi)$ is not designated for all classical valuations $w \sqsubseteq v$.

  The second claim has a dual proof. Each antitheorem of $\K$ is an anti\-theorem of $\CL$. Conversely, suppose that $v: \FmAlgebra \rightarrow \Kmatrix$ is a valuation such that $v(\varphi)$ is designated. Then in particular $v(\varphi)$ is a maximal element of the information order on $\Kmatrix$. The same argument as above shows that $w(\varphi)$ is designated for all classical valuations $w \sqsupseteq v$.
\end{proof}

  The admissibility of Cut and the antiadmissibility of Identity are now immediate consequences of the previous three propositions. We emphasize again that this route to proving the admissibility of Cut in the Gentzen calculus for classical logic was already taken by Pynko \cite{pynko10}. We therefore only provide a proof of the latter assertion.

\begin{theorem}[Admissibility of Cut] \label{thm:cut-admissible}
  In the Gentzen calculus which contains Identity, the common structural rules  of Weakening and Contraction, and the intro\-duction rules for all connectives, the Cut rule is admissible.
\end{theorem}

\begin{theorem}[Antiadmissibility of Identity] \label{thm:identity-antiadmissible}
  In\:the\:Gentzen\:calculus\:which contains Cut, the common structural rules, and the elimination rules for all connectives, the Identity rule is antiadmissible.
\end{theorem}

\begin{proof}
  Suppose that the empty sequent is derivable from the set of sequents $\SequentSet$ in the Gentzen calculus which containts Cut, the common structural rules, the elimination rules, and the Identity rule. Then it is derivable from $\SequentSet$ in $\GKforCL$. Using the equivalence between Hilbert and Gentzen versions of super-Belnap logics, Proposition \ref{prop:weak-conservativity} implies that the empty sequent is derivable from $\SequentSet$ in $\GKforK$. Proposition \ref{prop:intro-admissible} then implies that the empty sequent is derivable from $\SequentSet$ in the calculus which contains Cut, the common structural rules, and the elimination rules.
\end{proof}

  These proofs, based on Proposition \ref{prop:weak-conservativity}, of course do not yield a procedure for eliminating Cut or Identity from a given proof.

  In the rest of this section, we translate some of the results of \cite{prenosil16} into the language of super-Belnap calculi. Claims about super-Belnap logics made without proof in the rest of this section are established in \cite{prenosil16}.

  Theorem \ref{thm:cut-admissible} generalizes to arbitrary super-Belnap calculi, even though Propositions \ref{prop:elim-admissible} and \ref{prop:intro-admissible} are of course specific to $\LP$ (and $\LP \vee \ECQ$). Cut is the strongest non-axiomatic rule in super-Belnap logics, therefore the folowing theorem completely settles the question of admissibility of non-axiomatic rules in Gentzen calculi for super-Belnap logics.

\begin{theorem} \label{thm:cut-always-admissible}
  Cut is admissible in each super-Belnap calculus.
\end{theorem}

\begin{proof}
  The logics $\B$ and $\K$ have the same theorems, as do $\LP$ and $\CL$. Moreover, each super-Belnap logic $\logic{L}$ is either below $\K$ or above $\LP$. In the former case, $\logic{L} \vee \K = \K$, while in the latter case $\logic{L} \vee \K = \CL$. In both cases $\logic{L}$ and $\logic{L} \vee \K$ have the same theorems, i.e.\ the same sequents are provable in each super-Belnap calculus for $\logic{L}$ and in its extension by Cut.
\end{proof}

  When it comes to antiadmissibility, there is no such single theorem to cover all situations. Let us therefore only consider one simple example.

\begin{proposition}
  Limited Cut is antiadmissible in the calculus $\GKforECQ$ which extends $\GKforB$ by Explosive Cut.
\end{proposition}

\begin{proof}
  $\ECQ$ and $\ETL$ have the same antitheorems.
\end{proof}

  It is \emph{not} the case that Identity is antiadmissible in all super-Belnap cal\-culi, in contrast to Theorem \ref{thm:cut-always-admissible}. In particular, it is not antiadmissible in either of the calculi $\GKforECQ$ and $\GKforETL$, since the logics $\LP \vee \ECQ$ and $\CL = \LP \vee \ETL$ have more antitheorems than $\ECQ$ and $\ETL$, one example being the formula $(p \wedge \dmneg p) \vee (q \wedge \dmneg q)$.

\section{Analytic--synthetic proofs and structural \mbox{atomicity}}
\label{sec:analytic--synthetic}

  The Gentzen-style proof theory of classical logic has mainly been concerned with which sequents are provable, meaning provable from an empty set of premises. Accordingly, its ``Hauptsatz'' states that with an empty set of premises we may restrict to proofs which do not contain Cut.

  In the current context of super-Belnap calculi, we are mainly interested in proving sequents from other sequents. We would therefore like to formulate an appropriate generalization of this result which would cover proofs from non-empty sets of premises. Since in classical logic we may be interested in proving sequents from other sequents as well, such a generalization may be of interest even to the classical logician. In particular, we shall use it in the following section to provide an alternative syntactic proof of (a certain refinement of) the Craig interpolation theorem for classical logic.

  We propose to generalize the notion of a cut-free proof to proofs from a non-empty set of premises by decomposing this notion into a conjunction of two distinct conditions: structural atomicity and analyticity--syntheticity. The former notion concerns only the applications of structural rules in the proof, whereas the latter notion concerns only logical rules.

\begin{definition}
  A proof is \emph{structurally atomic} if both the premises and conclusions of all occurrences of structural rules in the proof are atomic sequents. A proof is \emph{analytic--synthetic} if in each branch of the proof all instances of elimination rules precede all instances of introduction rules.
\end{definition}

  Structurally atomic analytic--synthetic proofs may therefore be divided into three parts: a part consisting of elimination rules at the top, a part consisting of atomic instances of structural rules in the middle, and a part consisting of introduction rules at the bottom (each of these parts may be empty). The importance of structural atomicity is precisely that it yields this tripartite structure in conjunction with analyticity--syntheticity. Note that, as the following lemma shows, ``local'' analyticity--syntheticity implies ``global'' analyticity--syntheticity in the presence of structural atomicity.

\begin{lemma} \label{lemma:immediately-follows}
  A structurally atomic proof is analytic--synthetic if and only if no elimination rule in it immediately follows an introduction rule.
\end{lemma}

\begin{proof}
  Elimination rules may not immediately follow and introduction rules may not immediately precede any occurrence of a structural rule in a structurally atomic proof. Therefore if in some branch of the proof an instance of an introduction rule precedes an instance of an elimination rule, then there must be a pair of rules in between these two which consists of an introduction rule followed by an elimination rule.
\end{proof}

  It is worth observing that if we restrict to classical logic (i.e.\ if the specific structural rules are Identity and Cut) and to an empty set of premises, structurally atomic analytic--synthetic proofs are essentially ordinary cut-free proofs: clearly no elimination rules may occur in such proofs, and all instances of Cut are restricted to atomic sequents. The following proposition, whose proof is immediate, is now all it takes to transfom structurally atomic analytic--synthetic proofs from an empty set of premises in $\GKforCL$ into cut-free proofs in the standard sense of the term. Note that by an \emph{atomic version} of a rule, we mean the restriction of the rule to inferences in which all of the premises as well as the conclusion are atomic sequents.

\begin{proposition} \label{prop:initial-cuts-redundant}
  The following are equivalent for atomic sequents $\Gamma \seq \Delta$:
\begin{itemize}
\item[(i)] $\Gamma \seq \Delta$ has the form $p, \Gamma' \seq \Delta', p$.
\item[(ii)] $\Gamma \seq \Delta$ is derivable using atomic versions of Identity and Weakening.
\item[(iii)] $\Gamma \seq \Delta$ is derivable using atomic versions of Identity, Cut, Weakening, and Contraction.
\end{itemize}
\end{proposition}

  We may dualize this observation and simplify the structure of structurally atomic analytic--synthetic proofs of the empty sequent in an analogical way.

\begin{proposition}
  The following are equivalent for sets of atomic sequents~$\SequentSet$:
\begin{itemize}
\item[(i)] $\emptyset \seq \emptyset$ is derivable from $\SequentSet$ using atomic instances of Contraction followed by atomic instances of Cut.
\item[(ii)] $\emptyset \seq \emptyset$ is derivable from $\SequentSet$ using atomic versions of Contraction, Weakening, Identity, and Cut.
\end{itemize}
\end{proposition}

\begin{proof}
  Suppose that (ii) holds. We first show that Weakening is not needed. Observe that each atomic instance of Weakening by $p$ on the right (on the left) may be permuted below each immediately following atomic instance of Cut where $p$ is not the cut formula on the right (on the left) of the appropriate premise of Cut and below each immediately following atomic instance of Contraction where $p$ is not the contracted formula on the right (on the left). On the other hand, an atomic instance of Weakening by $p$ on the right (on the left) followed by an atomic instance of Cut where $p$ is the cut formula on the right (on the left) of the appropriate premise of Cut may be replaced by several atomic instances of Weakening. That is, the proof segment
\begin{prooftree}
\def\fCenter{\seq}
\Axiom$\Gamma \fCenter \Delta$
\UnaryInf$\Gamma \fCenter \Delta, p$
\Axiom$p, \Gamma' \fCenter \Delta'$
\BinaryInf$\Gamma, \Gamma' \fCenter \Delta, \Delta'$
\end{prooftree}
may be replaced by the proof segment
\begin{prooftree}
\def\fCenter{\seq}
\Axiom$\Gamma \fCenter \Delta$
\UnaryInf$\Gamma, \Gamma' \fCenter \Delta, \Delta'$
\end{prooftree}
where we have condensed several instances of Weakening into one step, and likewise for Weakening on the left. Similarly, a proof segment consisting of an atomic instance of Weakening by $p$ on the right (on the left) followed by an atomic instance of Contraction where $p$ is the contracted formula on the right (on the left) may simply be omitted from the proof. Since Weakening cannot be the last rule in a proof of the empty sequent, it follows that every instance of Weakening (starting with the bottommost ones) in a proof of the empty sequent from $\SequentSet$ which only uses atomic versions of Identity, Cut, Weakening, and Contraction may be permuted downward until it is removed from the proof. This yields a proof which only uses atomic versions of Identity, Cut, and Contraction. Moreover, each instance of Identity in such a proof may only be followed by an instance of Cut, which, however, is then clearly redundant. Thus we obtain a proof which only uses atomic versions of Cut and Contraction. Finally, each instance of Contraction may be permuted above each instance of Cut, yielding a proof which consists of atomic instances of Contraction followed by atomic instances of Cut.
\end{proof}

  We find it somewhat remarkable that although elimination rules have no place in the standard Gentzen calculi for classical logic, cut-free proofs arise naturally as the intersection of two classes of proofs defined in terms of elimination rules and atomic sequents.

  Transforming a structurally atomic proof into one which is moreover analytic--synthetic is not difficult, as the following proposition shows.

\begin{proposition} \label{prop:analytic--synthetic}
  If a sequent has a structurally atomic proof from a given set of sequents in a super-Belnap calculus, then it has a structurally atomic analytic--synthetic proof.
\end{proposition}

\begin{proof}
  By Lemma \ref{lemma:immediately-follows}, it suffices to produce a structurally atomic proof in which no elimination rule immediately follows an introduction rule. We shall first deal with finite proofs.

  Let us call an instance of an elimination rule \emph{problematic} if it immediately follows an instance of an introduction rule. The \emph{depth} of a given occurrence of a rule will be the length of the longest branch of the subproof which ends this rule. It suffices to show that a finite structurally atomic proof which contains exactly one problematic rule of depth $d$ and it is the final rule of the proof may be reduced to a finite structurally atomic proof which either contains no problematic rules or it contains exactly one problematic rule and it has height lower than $d$.

  There are only two cases: either the formula being broken down by the problematic rule is a side formula of the introduction rule above or it is the principal formula of the introduction rule above. In the former case, it is straightforward to permute the problematic rule above the introduction rule, thereby either decreasing its depth or making it unproblematic. In the latter case, the problematic rule may be eliminated directly. We again only consider the case of conjunction, the case of disjunction being dual and the cases of negation and the truth constants being simpler. The required reduction are straightforward: the proof segment
\begin{prooftree}
\def\fCenter{\seq}
\Axiom$\varphi, \psi, \Gamma \fCenter \Delta$
\UnaryInf$\varphi \wedge \psi \fCenter \Delta$
\UnaryInf$\varphi, \psi \fCenter \Delta$
\end{prooftree}
is replaced simply by 
\begin{prooftree}
\def\fCenter{\seq}
\Axiom$\varphi, \psi, \Gamma \fCenter \Delta$
\end{prooftree}
while the proof segments
\begin{prooftree}
\def\fCenter{\seq}
\def\defaultHypSeparation{\hskip \HypSep}
\Axiom$\Gamma \fCenter \Delta, \varphi$
\Axiom$\Gamma \fCenter \Delta, \psi$
\BinaryInf$\Gamma \fCenter \Delta, \varphi \wedge \psi$
\UnaryInf$\Gamma \fCenter \Delta, \varphi$

\Axiom$\Gamma \fCenter \Delta, \varphi$
\Axiom$\Gamma \fCenter \Delta, \psi$
\BinaryInf$\Gamma \fCenter \Delta, \varphi \wedge \psi$
\UnaryInf$\Gamma \fCenter \Delta, \psi$

\noLine
\BinaryInfC{}
\end{prooftree}
are replaced simply by
\begin{prooftree}
\def\fCenter{\seq}
\def\defaultHypSeparation{\hskip \LargerHypSep}
\Axiom$\Gamma \fCenter \Delta, \varphi$
\Axiom$\Gamma \fCenter \Delta, \psi$
\noLine
\BinaryInfC{}
\end{prooftree}
It is moreover clear that these reductions preserve structural atomicity.

  It remains to deal with proofs which are not finite. We do so by breaking them into finite parts. By a \emph{non-structural segment} of a proof, we shall mean a maximal subproof which does not contain any structural rules (a subproof being a suitably labelled subtree of a proof). That is, a non-structural segment is a subproof which (i) does not contain any structural rules, (ii) its root is either the conclusion of the whole proof or the premise of a structural rule, and (iii) its terminal nodes are either premises of the whole proof or conclusions of a structural rule. Each sequent in the proof belongs to some non-structural segment (possibly with only one node) and each non-structural segment is a finite proof, since it is finitely branching (by virtue of not containg any structural rules) and it does not contain an infinite branch (by virtue of being a subtree of a well-founded tree). Each non-structural segment may be assigned a finite \emph{structural height}, defined as the number of occurrences of structural rules which occur below its conclusion.

  We now transform the original proof into an analytic--synthetic proof in $\omega$ stages while preserving structural atomicity. In stage $0$, we transform the non-structural segment of structural height $0$ into an analytic--synthetic proof and append the appropriate subproofs of the original proof above the premises of this non-structural segment which are not premises of the original proof. In stage $n+1$, we transform each non-structural segment of structural height $n+1$ of the proof obtained after stage $n$ into an analytic--synthetic proof. Note that after stage $n$, each non-structural segment of structural height $m > n$ is in fact a non-structural segment of structural height $m$ in the original proof. In stages $m > n$, the non-structural segments of height at most $n$ are left unchanged. The limit case of this process is a suitably labelled tree in which each non-structural segment of structural height $n$ is precisely as it was after stage $n$.

  Suppose that this tree has an infinite branch. In particular, this branch contains infinitely many instances of structural rules. Now observe that two instances of structural rules are only connected by a branch in the limit stage if they were already connected at some finite stage, and they are already connected before stage $n+1$ if they were already connected before stage $n$. There is therefore a branch in the original proof containing infinitely many instances of structural rules. Since this cannot be the case, the limit stage in fact yields a well-founded tree and therefore a proof.
\end{proof}

  In order to transform each proof in a given calculus into a structurally atomic analytic--synthetic one, it therefore suffices to reduce every instance of a structural rule of the calculus into a proof which only contains logical rules and atomic instances of rules of the calculus. Rules which admit such reductions will be said to enjoy the expansion property, which is essentially nothing but the syntactic propagation property of Terui \cite{terui07}.

\begin{definition}
  A set of structural rules $R$ satisfies the \emph{expansion property} if the conclusion of each instance of a rule $\rho \in R$ has a proof from the corresponding instances of premises of $\rho$ which only uses the logical rules and atomic instances of rules in $R$.
\end{definition}

  A super-Belnap calculus satisfies the expansion property if the set of its structural rules does. A rule $\rho$ satisfies the expansion property if $\{ \rho \}$ does.

\begin{proposition} \label{prop:expansion-property}
  The structural rules of a super-Belnap calculus satisfy the expansion property if and only if for each proof of a sequent $\Gamma \seq \Delta$ from a set of sequents $\SequentSet$ in the calculus, there is a structurally atomic proof of $\Gamma \seq \Delta$ from $\SequentSet$.
\end{proposition}

\begin{proof}
  In the left-to-right direction, it suffices to replace each non-atomic instance of a structural rule with a suitable structurally atomic proof. Conversely, let $\rho$ be a structural rule of the calculus. Proposition~\ref{prop:analytic--synthetic} implies that there is a structurally atomic analytic--synthetic proof of the (atomic) conclusion of $\rho$ from the (atomic) premises of $\rho$. But such a proof cannot contain any logical rules.
\end{proof}

  The structural rules of $\GKforCL$ satisfy the expansion property.

\begin{proposition} \label{prop:structurally-atomic-specific}
  Identity, Weakening, and Contraction satisfy the expansion property. So does the set of rules $\{ \text{Cut}, \text{Contraction} \}$.
\end{proposition}

\begin{proof}
  In all cases this may be proved by induction over the complexity of the main formula of the rule, i.e.\ the formula denoted $\varphi$ in Figures \ref{fig:b-calculus} and \ref{fig:specific-structural-rules}.

  The proof for Identity is well known. For Weakening, the proof segment
\begin{prooftree}
\def\fCenter{\seq}
\Axiom$\Gamma \fCenter \Delta$
\UnaryInf$\varphi \wedge \psi, \Gamma \fCenter \Delta$
\end{prooftree}
is replaced by
\begin{prooftree}
\def\fCenter{\seq}
\Axiom$\Gamma \fCenter \Delta$
\UnaryInf$\varphi, \Gamma \fCenter \Delta$
\UnaryInf$\varphi, \psi, \Gamma \fCenter \Delta$
\UnaryInf$\varphi \wedge \psi, \Gamma \fCenter \Delta$
\end{prooftree}
while the proof segment
\begin{prooftree}
\def\fCenter{\seq}
\Axiom$\Gamma \fCenter \Delta$
\UnaryInf$\Gamma \fCenter \Delta, \varphi \wedge \psi$
\end{prooftree}
is replaced by
\begin{prooftree}
\def\fCenter{\seq}
\def\defaultHypSeparation{\hskip \HypSep}
\Axiom$\Gamma \fCenter \Delta$
\UnaryInf$\Gamma \fCenter \Delta, \varphi$
\Axiom$\Gamma \fCenter \Delta$
\UnaryInf$\Gamma \fCenter \Delta, \psi$
\BinaryInf$\Gamma \fCenter \Delta, \varphi \wedge \psi$
\end{prooftree}
For Contraction, the proof segment
\begin{prooftree}
\def\fCenter{\seq}
\Axiom$\varphi \wedge \psi, \varphi \wedge \psi, \Gamma \fCenter \Delta$
\UnaryInf$\varphi \wedge \psi, \Gamma \fCenter \Delta$
\end{prooftree}
is replaced by
\begin{prooftree}
\def\fCenter{\seq}
\Axiom$\varphi \wedge \psi, \varphi \wedge \psi, \Gamma \fCenter \Delta$
\UnaryInf$\varphi, \psi, \varphi \wedge \psi, \Gamma \fCenter \Delta$
\UnaryInf$\varphi, \psi, \varphi, \psi, \Gamma \fCenter \Delta$
\UnaryInf$\varphi, \psi, \psi, \Gamma \fCenter \Delta$
\UnaryInf$\varphi, \psi, \Gamma \fCenter \Delta$
\UnaryInf$\varphi \wedge \psi, \Gamma \fCenter \Delta$
\end{prooftree}
while the proof segment
\begin{prooftree}
\def\fCenter{\seq}
\Axiom$\Gamma \fCenter \Delta, \varphi \wedge \psi, \varphi \wedge \psi$
\UnaryInf$\Gamma \fCenter \Delta, \varphi \wedge \psi$
\end{prooftree}
is replaced by
\begin{prooftree}
\def\fCenter{\seq}
\def\defaultHypSeparation{\hskip \HypSep}
\Axiom$\Gamma \fCenter \Delta, \varphi \wedge \psi, \varphi \wedge \psi$
\UnaryInf$\Gamma \fCenter \Delta, \varphi, \varphi \wedge \psi$
\UnaryInf$\Gamma \fCenter \Delta, \varphi, \varphi$
\UnaryInf$\Gamma \fCenter \Delta, \varphi$
\Axiom$\Gamma \fCenter \Delta, \varphi \wedge \psi, \varphi \wedge \psi$
\UnaryInf$\Gamma \fCenter \Delta, \psi, \varphi \wedge \psi$
\UnaryInf$\Gamma \fCenter \Delta, \psi, \psi$
\UnaryInf$\Gamma \fCenter \Delta, \psi$
\BinaryInf$\Gamma \fCenter \Delta, \varphi \wedge \psi$
\end{prooftree}
Finally, in the case of Cut the proof segment
\begin{prooftree}
\def\fCenter{\seq}
\def\defaultHypSeparation{\hskip \HypSep}
\Axiom$\Gamma \fCenter \Delta, \varphi \wedge \psi$
\Axiom$\varphi \wedge \psi, \Gamma' \fCenter \Delta'$
\BinaryInf$\Gamma, \Gamma' \fCenter \Delta, \Delta'$
\end{prooftree}
is replaced by
\begin{prooftree}
\def\fCenter{\seq}
\def\defaultHypSeparation{\hskip \HypSep}
\Axiom$\Gamma \fCenter \Delta, \varphi \wedge \psi$
\UnaryInf$\Gamma \fCenter \Delta, \psi$

\Axiom$\Gamma \fCenter \Delta, \varphi \wedge \psi$
\UnaryInf$\Gamma \fCenter \Delta, \varphi$

\Axiom$\varphi \wedge \psi, \Gamma' \fCenter \Delta'$
\UnaryInf$\varphi, \psi, \Gamma' \fCenter \Delta'$

\BinaryInf$\psi, \Gamma, \Gamma' \fCenter \Delta, \Delta'$
\BinaryInf$\Gamma, \Gamma, \Gamma' \fCenter \Delta, \Delta, \Delta'$
\UnaryInf$\Gamma, \Gamma' \fCenter \Delta, \Delta'$
\end{prooftree}
where the last step condenses several instances of Contraction.

  The other connectives are again either dual or simpler to handle.
\end{proof}

\begin{corollary}
  If a sequent has a proof from a given set of sequents in one of the calculi $\GKforB$, $\GKforLP$, $\GKforK$, or $\GKforCL$, then it has a structurally atomic analytic--synthetic proof.
\end{corollary}

\begin{proof}
  This follows from Propositions \ref{prop:analytic--synthetic} and \ref{prop:structurally-atomic-specific}.
\end{proof}

  The above procedure for reducing structural rules to atomic structural rules and then transforming the resulting proof into a proof which is moreover analytic--synthetic subsumes a cut elimination procedure for classical propositional logic in view of Proposition \ref{prop:initial-cuts-redundant}.

  Several remarks are now in order concerning this procedure. Firstly, it illustrates again that the admissibility of Cut is related to the problem of reducing arbitrary instances of structural rules to atomic instances, as observed already by Terui~\cite{terui07}. Secondly, it shows that admitting elimination rules in a Gentzen calculus may be useful even if we are only interested in sequents provable from an empty set of premises. Although the final result of this transformation applied to such a proof does not contain any instances of elimination rules, the intermediate reductions do. And thirdly, it is worth noting explicitly that the notion of a structurally atomic analytic--synthetic proof does not single out any particular structural rule for special treatment, unlike the notion of a cut-free proof.

  Things need not always go as smoothly as in the case of Identity and Cut, of course. For example, the Explosive Cut rule cannot be replaced by its atomic version. It would be tempting to reduce the proof segment
\begin{prooftree}
\def\fCenter{\seq}
\Axiom$\emptyset \fCenter \varphi \wedge \psi$
\Axiom$\varphi \wedge \psi \fCenter \emptyset$
\BinaryInf$\emptyset \fCenter \emptyset$
\end{prooftree}
to the following would-be proof segment
\begin{prooftree}
\def\fCenter{\seq}
\Axiom$\emptyset \fCenter \varphi \wedge \psi$
\UnaryInf$\emptyset \fCenter \psi$

\Axiom$\emptyset \fCenter \varphi \wedge \psi$
\UnaryInf$\emptyset \fCenter \varphi$

\Axiom$\varphi \wedge \psi \fCenter \emptyset$
\UnaryInf$\varphi, \psi \fCenter \emptyset$

\BinaryInf$\psi \fCenter \emptyset$
\BinaryInf$\emptyset \fCenter \emptyset$
\end{prooftree}
but of course the right-hand premise of the final inference was not obtained by an Explosive Cut. The reader may verify that if $\varphi$ and $\psi$ are distinct atoms, then the above instance of Explosive Cut is in fact not derivable in any way in $\GKforB$ from the atomic version of the rule.

  To handle such situations, we use brute force: we simply add all the missing specific structural rules to the calculus. This may seem like cheating at first sight, but we shall see in the following section that even this brute force solution may be exploited to obtain useful results if we have a useful syntactic description of these rules, which we shall now provide.

  We call a set of sequents \emph{elimination-derivable (introduction-derivable)} from a set of sequents $\SequentSet$ if each sequent in it is derivable from $\SequentSet$ using only elimination rules (only introduction rules). Let $\AtSeq{\Gamma \seq \Delta}$ denote the set of all atomic sequents elimination-derivable from the sequent $\Gamma \seq \Delta$. We shall use the notation $\sigma(\Gamma \seq \Delta)$ for $\sigma[\Gamma] \seq \sigma[\Delta]$.

\begin{lemma} \label{lemma:elimination-derivable} ~
\begin{itemize}
\item[(i)] An atomic sequent is elimination-derivable from $\varphi \wedge \psi, \Gamma \seq \Delta$ if and only if it is elimination-derivable from $\varphi, \psi, \Gamma \seq \Delta$.
\item[(ii)] An atomic sequent is elimination-derivable from $\Gamma \seq \Delta, \varphi \wedge \psi$ if and only if it is elimination-derivable from either $\Gamma \seq \Delta, \varphi$ or $\Gamma \seq \Delta, \psi$.
\item[(iii)] An atomic sequent is elimination-derivable from $\varphi \vee \psi, \Gamma \seq \Delta$ if and only if it is elimination-derivable from either $\varphi, \Gamma \seq \Delta$ or $\psi, \Gamma \seq \Delta$.
\item[(iv)] An atomic sequent is elimination-derivable from $\Gamma \seq \Delta, \varphi \vee \psi$ if and only if it is elimination-derivable from $\Gamma \seq \Delta, \varphi, \psi$.
\item[(v)] An atomic sequent is elimination-derivable from $\dmneg \varphi, \Gamma \seq \Delta$ if and only if it is elimination-derivable from $\Gamma \seq \Delta, \varphi$.
\item[(vi)] An atomic sequent is elimination-derivable from $\Gamma \seq \Delta, \dmneg \varphi$ if and only if it is elimination-derivable from $\varphi, \Gamma \seq \Delta$.
\item[(vii)] An atomic sequent is elimination-derivable from $\top, \Gamma \seq \Delta$ if and only if it is elimination-derivable from $\Gamma \seq \Delta$.
\item[(viii)] No atomic sequent is elimination-derivable from $\Gamma \seq \Delta, \top$.
\item[(ix)] No atomic sequent is elimination-derivable from $\bot, \Gamma \seq \Delta$.
\item[(x)] An atomic sequent is elimination-derivable from $\Gamma \seq \Delta, \bot$ if and only if it is elimination-derivable from $\Gamma \seq \Delta$.
\end{itemize}
\end{lemma}

\begin{proof}
  The right-to-left implications are trivial. To prove the converse implications, it suffices to replace $\emptyset \seq \emptyset$ by a given atomic sequent in the proof of Proposition \ref{prop:intro-admissible}.
\end{proof}

\begin{lemma} \label{lemma:atomic-reduction}
  The sequent $\Gamma \seq \Delta$ is introduction-derivable from $\AtSeq{\Gamma \seq \Delta}$.
\end{lemma}

\begin{proof}
  The claim can be proved by a straightforward induction on the complexity of $\Gamma \seq \Delta$.
\end{proof}

\begin{lemma} \label{lemma:subst1}
  $\AtSeq{\sigma (\Gamma \seq \Delta)}$ is elimination-derivable from $\sigma[\AtSeq{\Gamma \seq \Delta}]$.
\end{lemma}

\begin{proof}
  We prove the claim by induction over the complexity of the sequent $\Gamma \seq \Delta$, i.e.\ the number of connectives in it. If $\Gamma \seq \Delta$ is atomic, the claim holds trivially by the definition of $\AtSeq{\sigma (\Gamma \seq \Delta)}$. Now consider sequents of the form $\Gamma \seq \Delta, \varphi \wedge \psi$. Lemma~\ref{lemma:elimination-derivable} yields that $\AtSeq{\sigma(\Gamma \seq \Delta, \varphi \wedge \psi)} = \AtSeq{\sigma(\Gamma \seq \Delta, \varphi)} \cup \AtSeq{\sigma(\Gamma \seq \Delta, \psi)}$. By the induction hypothesis both sets of sequents $\AtSeq{\sigma(\Gamma \seq \Delta, \varphi)}$ and $\AtSeq{\sigma(\Gamma \seq \Delta, \psi)}$ are elimination-derivable from $\sigma[\AtSeq{\Gamma \seq \Delta, \varphi}]$ and $\sigma[\AtSeq{\Gamma \seq \Delta, \psi}]$, respectively. Moreover, $\AtSeq{\Gamma \seq \Delta, \varphi} \cup \AtSeq{\Gamma \seq \Delta, \psi} \subseteq \AtSeq{\Gamma \seq \Delta, \varphi \wedge \psi}$, therefore the set of sequents $\AtSeq{\sigma(\Gamma \seq \Delta, \varphi \wedge \psi)}$ is elimination-derivable from the set of sequents $\sigma[\AtSeq{\Gamma \seq \Delta, \varphi \wedge \psi}]$. The remaining cases are either analogical or simpler.
\end{proof}

\begin{lemma} \label{lemma:subst2}
  $\sigma[\AtSeq{\Gamma \seq \Delta}]$ is introduction-derivable from $\AtSeq{\sigma[\Gamma \seq \Delta]}$.
\end{lemma}

\begin{proof}
  We prove the claim by induction over the complexity of $\Gamma \seq \Delta$. If $\Gamma \seq \Delta$ is atomic, we are to show that $\sigma(\Gamma \seq \Delta)$ is introduction-derivable from $\AtSeq{\sigma(\Gamma \seq \Delta)}$. But this holds by Lemma~\ref{lemma:atomic-reduction}. Now consider sequents of the form $\Gamma \seq \Delta, \varphi \wedge \psi$. By Lemma \ref{lemma:elimination-derivable} we know that $\AtSeq{\Gamma \seq \Delta, \varphi \wedge \psi} = \AtSeq{\Gamma \seq \Delta, \varphi} \cup \AtSeq{\Gamma \seq \Delta, \varphi}$ and likewise for $\AtSeq{\sigma(\Gamma \seq \Delta, \varphi \wedge \psi)}$. The claim now follows immediately from the induction hypothesis. The remaining cases are either analogical or simpler.
\end{proof}

  Given a structural rule $\rho$, we now provide a syntactically defined set of rules which contains $\rho$, satisfies the expansion property, and moreover each of the rules is valid in each logic which validates $\rho$.

\begin{definition}
  Let $\set{\Gamma_{i} \seq \Delta_{i}}{i \in I} \vdash \Gamma \seq \Delta$ be a structural rule and $\sigma$ be a substitution. Then a \emph{$\sigma$-expansion} of this structural rule is a structural rule of the form $\bigcup_{i \in I} \AtSeq{\sigma(\Gamma_{i} \seq \Delta_{i})} \vdash \Lambda \seq \Pi$ for $\Lambda \seq \Pi$ in $\AtSeq{\sigma(\Gamma \seq \Delta)}$.
\end{definition}

  For example, the left-hand version of the Limited Cut rule in Figure \ref{fig:specific-structural-rules} is a schema standing for a set of structural rules which contains the rule
\begin{prooftree}
\def\fCenter{\seq}
\Axiom$\emptyset \fCenter p$
\Axiom$p, r\fCenter s$
\BinaryInf$r \fCenter s$
\end{prooftree}
whose $\sigma$-expansion for $\sigma(p) = p \wedge q$ and $\sigma(r) = r$ and $\sigma(s) = s$ is the rule
\begin{prooftree}
\def\fCenter{\seq}
\Axiom$\emptyset \fCenter p$
\Axiom$\emptyset \fCenter q$
\Axiom$p, q, r\fCenter s$
\TrinaryInf$r \fCenter s$
\end{prooftree}
interpreted as
\begin{prooftree}
\def\fCenter{\seq}
\Axiom$\emptyset \fCenter \varphi$
\Axiom$\emptyset \fCenter \psi$
\Axiom$\varphi, \psi, \Gamma \fCenter \Delta$
\TrinaryInf$\Gamma \fCenter \Delta$
\end{prooftree}
  The following observations are now immediate.

\begin{proposition} \label{prop:expansions-valid}
  If a structural rule is valid in a super-Belnap Gentzen relation, then so are all of its expansions.
\end{proposition}

\begin{proof}
  This claim follows immediately from the fact that $\AtSeq{\Gamma \seq \Delta}$ is equivalent in $\GB$ to $\Gamma \seq \Delta$ for each sequent $\Gamma \seq \Delta$ by Lemma \ref{lemma:atomic-reduction}.
\end{proof}

  We will in fact be interested in expansions of a particular kind. In the following definition, positive and negative occurrences of atoms are defined inductively as expected, e.g.\ the atom $p$ occurs negatively and the atom $q$ occurs positively in the formula $\dmneg p \vee q$.

\begin{definition}
  A formula is \emph{balanced} if each atom occurs only positively or only negatively in it. A substitution $\sigma$ is \emph{balanced} if the formula $\sigma(p)$ is balanced for each atom $p$. A substitution $\sigma$ is \emph{non-conflicting} if $\sigma(p)$ and $\sigma(q)$ do not share any variables for distinct atoms $p$ and $q$. A substitution~$\sigma$ is \emph{atomic} if $\sigma(p)$ is an atom for each atom $p$.
\end{definition}

\begin{lemma} \label{lemma:decomposition}
  Each substitution $\sigma$ is the composition $\sigma_{\text{sa}} \circ \sigma_{\text{bnc}}$ of a balanced non-conflicting substitution and a surjective atomic substitution.
\end{lemma}

\begin{proof}
  Let $\sigma_{p}$ and $\tau_{p}$ for each atom $p$ be atomic substitutions such that $(\tau_{p} \circ \sigma_{p})(q) = q$ for each atom $q$ and moreover the ranges of $\sigma_{p}$ and $\sigma_{q}$ are disjoint for distinct atoms $p$ and $q$. Let $\gamma_{p}$ and $\delta_{p}$ be atomic substitutions such that $(\delta_{p} \circ \gamma_{p})(q) = q$ for each $q$ and the ranges of $\gamma_{p}$ and $\gamma_{q}$ are distinct for distinct atoms $p$ and $q$ and moreover the ranges of $\gamma_{p}$ and $\sigma_{q}$ are disjoint for all atoms $p$~and~$q$. Suppose also that there are $\kappa$ atoms which do not lie in the range of any of the functions $\sigma_{p}$ or $\gamma_{p}$, where $\kappa$ is the cardinality of the set of all atoms. Such substitutions always exist, since each set of cardinality $\kappa$ may be decomposed into $\kappa$ disjoint subsets of cardinality~$\kappa$.

  Now given a substitution~$\sigma$ we define a non-conflicting substitution $\sigma_{\text{nc}}$ so that $\sigma_{\text{nc}}(p) = (\sigma_{p} \circ \sigma)(p)$. When then modify $\sigma_{\text{nc}}$ to obtain a balanced non-conflicting substitution $\sigma_{\text{bnc}}$ by changing each negative occurrence of a variable $q$ in $\sigma_{\text{nc}}(p)$ to $\gamma_{p}(q)$. The substitution $\sigma_{\text{sa}}$ may now be defined so that $\sigma_{\text{sa}}(q) = \tau_{p}(q)$ whenever $q$ is in the range of $\sigma_{p}$ and $\sigma_{\text{sa}}(q) = (\tau_{p} \circ \delta_{p})(q)$ whenever $q$ is in the (disjoint) range of $\gamma_{p}$. Moreover, we may define $\sigma_{\text{sa}}(q)$ for $q$ outside the ranges of these functions so that $\sigma_{\text{sa}}$ is a surjective atomic substitution.
\end{proof}

  A balanced (non-conflicting) expansion of a structural rule is defined as a $\sigma$-expansion of the rule for some balanced (non-conflicting) substitution~$\sigma$.

\begin{proposition} \label{prop:sigma-expansion-property}
  The set of all balanced non-conflicting expansions of a structural rule satisfies the expansion property.
\end{proposition}

\begin{proof}
  Let $\set{ \Gamma_{i} \seq \Delta_{i} }{i \in I} \vdash \Gamma \seq \Delta$ be a given structural rule. Let us write $\SequentSet \vdash_{\text{elim}} \SequentSet'$, $\SequentSet \vdash_{\text{intro}} \SequentSet'$, and $\SequentSet \vdash_{\text{at}} \SequentSet'$ to abbreviate the claims that each sequent in $\SequentSet'$ has a proof from $\SequentSet$ which uses, respectively, only the elimination rules, only the introduction rules, and only atomic instances of balanced non-conflicting expansions of the given rule.

  We are to show that $\bigcup_{i \in I} \tau[\AtSeq{\sigma(\Gamma_{i} \seq \Delta_{i})}] \vdash_{\text{at}} \tau[\AtSeq{\sigma(\Gamma \seq \Delta)}]$ for each instance given by $\tau$ of each $\sigma$-expansion of the given rule, where $\sigma$ is a balanced non-conflicting expansion.

  By Lemma~\ref{lemma:decomposition} there is a balanced non-conflicting substitution $\tau_{\text{bnc}}$ and a surjective atomic substitution $\tau_{\text{sa}}$ such that $\tau = \tau_{\text{sa}} \circ \tau_{\text{nbc}}$. Observe that $\tau_{\text{bnc}} \circ \sigma$ is also a balanced non-conflicting expansion.

  We have $(\tau_{\text{sa}}) \circ (\tau_{\text{nbc}})[\AtSeq{\sigma(\Gamma_{i} \seq \Delta_{i})}] \vdash_{\text{elim}} \tau_{\text{sa}} [\AtSeq{(\tau_{\text{nbc}} \circ \sigma)(\Gamma_{i} \seq \Delta_{i})}]$ by Lemma~\ref{lemma:subst1}. Also $\bigcup_{i \in I} \AtSeq{(\tau_{\text{bnc}} \circ \sigma)(\Gamma_{i} \seq \Delta_{i})} \vdash_{\text{at}}\AtSeq{(\tau_{\text{bnc}} \circ \sigma)(\Gamma \seq \Delta)}$. Applying~$\tau_{\text{sa}}$ to this consequence preserves the relation~$\vdash_{\text{at}}$. Lemma~\ref{lemma:subst2} now yields that $\tau_{\text{sa}}[\AtSeq{(\tau_{\text{bnc}} \circ \sigma)(\Gamma \seq \Delta)}] \vdash_{\text{intro}} (\tau_{\text{sa}} \circ \tau_{\text{nbc}}) [\AtSeq{\sigma(\Gamma \seq \Delta)}]$ and $\tau_{\text{sa}} \circ \tau_{\text{nbc}} = \tau$.
\end{proof}

  A trivial way of ensuring that a super-Belnap calculus satisfies the expansion property is therefore to add all balanced non-conflicting expansions of all structural rules to the calculus by fiat. Obtaining a reasonable description of the resulting set of rules is the only (optional) part of the transformation procedure from a Hilbert axiomatization to a well-behaved Gentzen calculus which is specific to each individual super-Belnap logic.

  Finally, recall that an important feature of cut-free proofs in the standard Gentzen calculus for classical logic is the subformula property. This property is useful when trying to reduce the infinite space of all possible proofs of a given sequent to a finite one (in a finitary calculus).

\begin{definition}
  A proof has the \emph{subformula property} if each formula of the proof is a subformula of some formula either in the premises or in the conclusion.
\end{definition}

  It is not the case that each structurally atomic analytic--synthetic proof has the subformula property. For example, we may apply Weakening by a variable $p$ which occurs neither in the premises nor in the conclusion to each side of an atomic sequent $\Gamma \seq \Delta$, and then Cut on $p$. However, it is easy to transform such a proof into one which has the subformula property. In the following proposition, by a \emph{non-trivial} super-Belnap calculus we mean a super-Belnap calculus for a non-trivial logic.

\begin{proposition} \label{prop:subformula property}
  If a sequent has a structurally atomic analytic--synthetic proof from $\SequentSet$ in a non-trivial super-Belnap calculus, then it has a structurally atomic analytic--synthetic proof from $\SequentSet$ with the subformula property.
\end{proposition}

\begin{proof}
  Suppose first that the premises do not contain any atom (as a subformula). Each such premise is easily seen to be equi\-valent in $\GB$ either to the empty sequent or to the empty set of sequents. In the former case, the empty sequent (from which every other sequent is provable using only the atomic form of Weakening followed by introduction rules) is provable from the constant sequent using only elimination rules. If each premise falls under the latter case, then the conclusion is provable from an empty set of premises. Moreover, each non-trivial super-Belnap logic has the same set of constant theorems as $\B$. Therefore if the conclusion does not contain any atom, then it is in fact derivable from an empty set of premises in $\GKforB$. By Propositions \ref{prop:analytic--synthetic}, \ref{prop:expansion-property}, and \ref{prop:structurally-atomic-specific} it has a structurally atomic analytic--synthetic proof from $\emptyset$ in $\GKforB$. But such a proof only uses introduction rules and the common structural rules, therefore it has the subformula property.

 Otherwise, let $q$ be an atom which occurs either in the premises of in the conclusion of the proof. Observe that each atom $p$ which occurs neither in the premises nor in the conclusion of a structurally atomic analytic--synthetic proof may only occur in atomic sequents, and the only rules which may apply to sequents containing $p$ are structural ones. Replacing all occurrences of $p$ by~$q$ now yields a structurally atomic analytic--synthetic proof which moreover has the subformula property.
\end{proof}

  Putting the results of this section together yields the following theorem.

\begin{theorem} \label{thm:normal-form}
  If a sequent has a proof from a set of sequents $\SequentSet$ in a super-Belnap calculus which satisfies the expansion property, then it has a structurally atomic analytic--synthetic proof from $\SequentSet$ with the subformula property.
\end{theorem}

\begin{proof}
  This folows from Propositions \ref{prop:analytic--synthetic}, \ref{prop:expansion-property}, and \ref{prop:subformula property}.
\end{proof}

\section{Interpolation in super-Belnap logics}
\label{sec:interpolation}

  We now apply the results proved in the previous section to obtain some interpolation theorems for super-Belnap logics. We say that a logic has the \emph{(simple) Craig interpolation property}, or briefly \emph{has interpolation}, if $\varphi \vdash_{\logic{L}} \psi$ implies the existence of a formula $\chi$ called the \emph{interpolant} of $\varphi$ and $\psi$ such that $\varphi \vdash_{\logic{L}} \chi$ and $\chi \vdash_{\logic{L}} \varphi$ and each atom which occurs in $\chi$ occurs in both $\varphi$ and $\psi$. We prove by a simple argument that the logics $\B$, $\K$, $\ETL$ and some others enjoy interpolation, in fact in a somewhat stronger form.

  Let $\logic{L}$, $\logic{L}_{1}$, and $\logic{L}_{2}$ be logics such that $\logic{L}_{1}, \logic{L}_{2} \subseteq \logic{L}$. We say that $\logic{L}$ enjoys \emph{$(\logic{L}_{1}, \logic{L}_{2})$-interpolation} if $\varphi \vdash_{\logic{L}} \psi$ implies the existence of an interpolant $\chi$ such that $\varphi \vdash_{\logic{L}_{1}} \chi$ and $\chi \vdash_{\logic{L}_{2}} \psi$ and each atom which occurs in $\chi$ occurs in both $\varphi$ and $\psi$. In that case $\logic{L} = \logic{L}_{1} \vee \logic{L}_{2}$. Clearly $(\logic{L},\logic{L})$-interpolation for $\logic{L}$ amounts precisely to interpolation for $\logic{L}$.

\begin{proposition} \label{prop:exp-interpolation}
  Let $\logic{L}$ be an extension of a logic $\logic{L}_{0}$ such that $\bot \vdash_{\logic{L}_{0}} p$ for some constant formula $\bot$. If $\logic{L}$ has inter\-polation or $(\logic{L}, \logic{L}_{0})$-interpolation, then so do all of its explosive extensions.
\end{proposition}

\begin{proof}
  Let $\logic{L}_{exp}$ be an explosive extension of $\logic{L}$. If $\varphi \vdash_{\logic{L}_{exp}} \psi$, then either $\varphi \vdash_{\logic{L}} \psi$ or $\varphi \vdash_{\logic{L}_{exp}} \emptyset$. In the former case the existence of the interpolant is guaranteed by the assumption that $\logic{L}$ has interpolation or $(\logic{L}, \logic{L}_{0})$-interpolation, in the latter case we may take $\bot$ as the interpolant.
\end{proof}

  Let us now review what is known about interpolation in super-Belnap logics. Interpolation for the Dunn--Belnap logic $\B$ was proved early on by Anderson and Belnap \cite[p.~161]{anderson+belnap75}. Interpolation for the strong three-valued Kleene logic $\K$ was proved using different techniques by Bendov\'{a} \cite{bendova05} and Milne \cite{milne16}.\footnote{Both of the papers \cite{bendova05} and \cite{milne16} in fact consider the fragment of $\K$ without the truth constants $\top$ and $\bot$, therefore they have to formulate the interpolation property with more care. However, ordinary interpolation for $\K$ is a straightforward consequence of their interpolation results. We consider this to be yet another reason to include the truth constants in the signature of super-Belnap logics. Likewise, Anderson and Belnap in fact consider the corresponding fragment of $\B$, although in their case no adjustments are needed in the definition of the interpolation property.} In the same paper, Milne also observed that interpolation for the Logic of Paradox $\LP$ is equi\-valent to interpolation for $\K$ in view of Lemma~\ref{lemma:contraposition}\,(ii), which he calls the Duality Principle. Moreover, he proved that classical logic $\CL$ enjoys $(\K, \LP)$-interpolation. Finally, Bendov\'{a} observed in her paper that the logic $\Kleq$ does not enjoy interpolation: the rule $(p \wedge \dmneg p) \vee r \vdash (q \vee \dmneg q) \vee r$ is valid in $\Kleq$ but lacks an interpolant. The same example shows that $\Kleq \vee \ECQ$ does not enjoy interpolation, either. As far as we are aware, this exhausts the present state of knowledge about inter\-polation in super-Belnap logics. To the best of our knowledge, interpolation has so far not been studied in other super-Belnap logics, which have only been introduced very recently.

\begin{lemma} \label{lemma:contraposition}
  The following equivalences hold:
\begin{itemize}
\item[(i)] $\varphi \vdash_{\B} \psi$ if and only if $\dmneg \psi \vdash_{\B} \dmneg \varphi$.
\item[(ii)] $\varphi \vdash_{\LP} \psi$ if and only if $\dmneg \psi \vdash_{\K} \dmneg \varphi$.
\end{itemize}
\end{lemma}

\begin{proof}
  The first equivalence is well known and the second was observed by Milne \cite{milne16}.
\end{proof}

\begin{lemma}[\cite{prenosil16}] \label{lemma:lp-cap-ecq}
  If $\logic{L}$ is a proper extension of $\B$, then $p, \dmneg p \vdash_{\logic{L}} q \vee \dmneg q$.
\end{lemma}

\begin{proposition} \label{prop:interpolation-lp-ecq}
  Let $\logic{L}$ be a proper extension of $\B$. If $\logic{L}$ enjoys inter\-polation, then either $\logic{L} = \LP$ or $\ECQ \subseteq \logic{L}$.
\end{proposition}

\begin{proof}
  By Lemma \ref{lemma:lp-cap-ecq} we have $p, \dmneg p \vdash_{\logic{L}} q \vee \dmneg q$. Only a constant formula may be an interpolant of this rule, and all constant formulas are equivalent in $\B$ to either $\bot$ or $\top$. Therefore either $p, \dmneg p \vdash_{\logic{L}} \bot$ or $\top \vdash_{\logic{L}} q \vee \dmneg q$. Moreover, Pynko \cite{pynko00} proved that each proper extension of $\LP$ lies above $\ECQ$.
\end{proof}

  Taking into account that there is a continuum of explosive extensions of~$\B$ as well as a continuum of logics in the interval $[\B, \LP]$ as proved in \cite{prenosil16}, we obtain the following corollary to Propositions \ref{prop:exp-interpolation} and \ref{prop:interpolation-lp-ecq}.

\begin{corollary}
  There is a continuum of super-Benap logics with inter\-polation and a continuum of super-Belnap logics without interpolation.
\end{corollary}

  The interpolation properties defined above extend naturally to Gentzen relations. To obtain the appropriate definitions for Gentzen relations, it suffices to replace the formulas $\varphi$, $\psi$, and $\chi$ by sequents.

\begin{proposition}
  Let $\logic{L}_{i}$ be a super-Belnap logic and $\GentzenRelation \logic{L}_{i}$ be the Gentzen relation simply equivalent to $\logic{L}_{i}$ via $\tautrans$ for $i \in \{ 0, 1, 2 \}$. Then $\logic{L}_{0}$ enjoys $(\logic{L}_{1}, \logic{L}_{2})$-interpolation if and only if $\GentzenRelation \logic{L}_{0}$ enjoys $(\GentzenRelation \logic{L}_{1}, \GentzenRelation \logic{L}_{2})$-interpolation.
\end{proposition}

\begin{proof}
  This holds by virtue of the fact that a variable occurs in $\tautrans(\Gamma \seq \Delta)$ or $\rhotrans(\varphi)$ respectively if and only if it occurs in $\Gamma \seq \Delta$ or $\varphi$.
\end{proof}

  Moreover, in the Gentzen case it suffices by Proposition \ref{prop:sequent-composition-decomposition} to find a \emph{set} of sequents which jointly plays the role of the interpolant.

  We now provide a broad sufficient condition for a super-Belnap logic $\logic{L}$ to enjoy $(\logic{L}, \B)$-interpolation and therefore ordinary interpolation.

\begin{definition}
  A \emph{cut formula} of a structural rule is a formula which only occurs in the premises of the rule. A \emph{side formula} of a structural rule is a formula which occurs only on the left-hand sides or only on the right-hand sides of sequents in the rule.
\end{definition}

  A cut formula (a side formula) of an instance of a structural rule is the appropriate instance of the atomic cut formula (the atomic side formula). It may happen, although this case is not very interesting, that a formula is both a cut formula and a side formula of a structural rule.

\begin{definition}
  A \emph{generalized cut rule} is a structural rule such that each formula which occurs in the rule is either a cut formula or a side formula.
\end{definition}

  For example, Limited Cut and Explosive Cut are generalized cut rules, whereas Identity and the rule for $\Kleq$ combining Identity and Cut are not.

\begin{definition}
  A rule \emph{does not introduce new variables} if all variables which occur in the conclusion also occur some of the premises.
\end{definition}

  In particular, a generalized cut rule does not introduce new variables, and neither do any of the elimination rules. As far as interpolation goes, Weakening is the only problematic rule which may introduce new variables.

\begin{proposition} \label{prop:generalized-cut-expansions}
  If a structural rule is a generalized cut rule, then so are all of its balanced non-conflicting expansions. If a structural rule does not introduce new variables, then neither do any of its expansions.
\end{proposition}

\begin{proof}
  Let $\set{\Gamma_{i} \seq \Delta_{i}}{i \in I} \vdash \Gamma \seq \Delta$ be a generalized cut rule (let us call it $\rho$) and let $\sigma$ be a balanced non-conflicting substitution. It is easy to prove that an atom which only occurs positively in a sequent $\Lambda \seq \Pi$ (i.e.\ only occurs positively in formulas in $\Pi$ and only occurs negatively in formulas in $\Lambda$) will only occur on the right-hand side of each sequent in $\AtSeq{\Lambda \seq \Pi}$, and likewise an atom which only occurs negatively in $\Lambda \seq \Pi$ will only occur on the left-hand side of each sequent in $\AtSeq{\Lambda \seq \Pi}$. Therefore if $p$ is a side formula in $\rho$, then each atom of $\sigma(p)$ will be a side formula of the $\sigma$-expansion of $\rho$, using the fact that $\sigma$ is balanced and non-conflicting. It is also easy to observe that $\AtSeq{\Gamma \seq \Delta}$ and $\Gamma \seq \Delta$ contain exactly the same atoms. Therefore if $p$ is a cut formula of $\rho$, then each atom of $\sigma(p)$ will also be a cut formula of the $\sigma$-expansion of $\rho$, using again the fact that $\sigma$ is non-conflicting. The same observation regarding the atoms which occur in $\AtSeq{\Gamma \seq \Delta}$ also proves the second claim.
\end{proof}

\begin{theorem} \label{thm:interpolation}
  Each extension $\GentzenRelation \logic{L}$ of $\GB$ by a set of generalized cut rules enjoys $(\GentzenRelation \logic{L}, \GB)$-interpolation.
\end{theorem}

\begin{proof}
  If a generalized cut rule is valid in $\GB$, then so are all of its balanced non-conflicting expansions by Proposition \ref{prop:expansions-valid}. By Proposition \ref{prop:generalized-cut-expansions} these are all generalized cut rules. Now consider the extension of $\GKforB$ by the set of all balanced non-conflicting expansions of all rules in the given set of generalized cut rules. This calculus satisfies the expansion property by Proposition \ref{prop:sigma-expansion-property}. By Theorem \ref{thm:normal-form} we may therefore restrict to structurally atomic analytic--synthetic proofs in this calculus. Moreover, Weakening is the only structural rule of this calculus which intro\-duces new variables.

  Consider a structurally atomic analytic--synthetic proof of $\Gamma \seq \Delta$ from the premises $\SequentSet$ in this calculus. Let us call a node in this proof \emph{critical} if all inferences above the node are elimination rules or structural rules and all inferences below are introduction rules. Each branch of the proof either intersects a critical node or terminates in a logical axiom.

  If each critical node only contains variables which occur in some premise of the proof, then the set of critical nodes forms an interpolant between $\SequentSet$ and $\Gamma \seq \Delta$, since each variable which occurs in a critical sequent also occurs in the conclusion. To prove the theorem, it therefore suffices to show that if $p$ does not occur in $\SequentSet$ and an atomic sequent $\Lambda \seq \Pi, p$ or $p, \Lambda \seq \Pi$ has a structurally atomic analytic--synthetic proof in the calculus, then so does $\Lambda \seq \Pi$. This is because each critical node $\Lambda \seq \Pi$ may then be transformed into a sequent $\Lambda' \seq \Pi'$ which only contains variables which occur in $\SequentSet$ by finitely many applications of this transformation, and moreover the sequent $\Lambda \seq \Pi$ is derivable from $\Lambda' \seq \Pi'$ using finitely many atomic instances of Weakening. This transformation may be performed on all critical nodes simultaneously, since no branch of the proof contains two such nodes.

  Thus, consider an atom $p$ which does not occur in $\SequentSet$ and a sequent $\Lambda \seq \Pi, p$ or $p, \Lambda \seq \Pi$ which has a structurally atomic analytic--synthetic proof from $\SequentSet$ in the calculus. The tree of all ancestors of this instance of $p$ is defined in the obvious way. The leaves of this tree must be the results of Weakening, since no other rule in the proof above $\Lambda \seq \Pi$ introduces new variables. Crucially, since all the specific structural rules are generalized cut rules, the atom in question must be a side formula of each specific structural rule in this tree. It is now immediate that we may delete all the ancestors of this instance of $p$ from the subproof, and obtain a (structurally atomic analytic--synthetic) proof of $\Lambda \seq \Pi$.
\end{proof}

  The following results are now immediate corollaries of the above theorem. The first claim of the following proposition was proved in~\cite{anderson+belnap75}, as recalled above, while the others appear to be new.

\begin{proposition} \label{prop:interpolation}
  The logic $\B$ enjoys interpolation. The logic $\K$ enjoys $(\K, \B)$-interpolation. The logic $\ETL$ enjoys $(\ETL, \B)$-interpolation.
\end{proposition}

\begin{proposition}
  The logic $\LP$ enjoys $(\B, \LP)$-interpolation.
\end{proposition}

\begin{proof}
  The claim follows immediately from Lemma \ref{lemma:contraposition}.
\end{proof}

\begin{proposition}
  The logic $\LP \vee \ECQ$ has $(\LP \vee \ECQ, \LP)$-interpolation.
\end{proposition}

\begin{proof}
  The logic $\LP$ enjoys $(\LP, \LP)$-interpolation and $\LP \vee \ECQ$ is an explosive extension of $\LP$. The claim now follows by Proposition \ref{prop:exp-interpolation}.
\end{proof}

  Interpolation for some other super-Belnap logics may also be established using Theorem \ref{thm:interpolation}, including the extensions of $\B$ or $\ETL$ by the rules
\begin{equation*}
  (p_{1} \wedge \dmneg p_{1}) \vee \ldots \vee (p_{n} \wedge \dmneg p_{n}) \vdash \emptyset,
\end{equation*}
  and the extensions of $\B$ by the rules
\begin{equation*}
  (p_{1} \wedge \dmneg p_{1}) \vee \ldots \vee (p_{n} \wedge \dmneg p_{n}) \vee q, \dmneg q \vee r \vdash r.
\end{equation*}
  It was proved in \cite{rivieccio12} and \cite{prenosil16} that these two rule schemas define strictly increasing sequences of extensions of $\B$ and $\ETL$.




  Moreover, a slight modification of Proposition \ref{prop:interpolation} yields a syntactic proof of Milne's non-classical refinement of the Craig inter\-polation theorem for classical propositional logic. The reader is encouraged to compare this proof with the standard syntactic proof of interpolation for classical logic based on the cut elimination theorem, found e.g.\ in \cite[Section 4.4.2]{basicprooftheory}

\begin{proposition}[\cite{milne16}] \label{prop:cl-interpolation}
  The logic $\CL$ enjoys $(\K, \LP)$-interpolation.
\end{proposition}

\begin{proof}
  Recall that the calculus $\GKforCL$ extends $\GKforB$ by Identity and Cut. We wish to separate each structurally atomic analytic--synthetic proof in this calculus into a part which contains no instances of Cut and a part which contains no instances of Identity. To this end, suppose that a branch of the proof contains an instance of an Identity rule followed by a Cut, and suppose that no other instances of Cut occur between these two rules. Then only the common structural rules may occur between these two rules, therefore one of the premises of the Cut has the form $p, \Gamma \seq \Delta, p$ by Proposition~\ref{prop:initial-cuts-redundant}. If the cut formula of such an instance of Cut is $p$, this instance of Cut may be replaced by Weakening. If it is some other formula, then the conclusion of the cut has the form $p, \Gamma' \seq \Delta', p$ and thus may be derived using Identity and Weakening only. Repeated applications of this transformation yield a structurally atomic analytic--synthetic proof in which there is no branch containing both Identity and Cut.

  Define the \emph{separating set} of this proof as the set of all sequents in the proof such that only elimination rules and structural rules other than Identity occur above them and only introduction rules occur below. Each branch of the proof either intersects the separating set or ends with an instance of Identity (or one of the axioms $\emptyset \seq \top$ or $\bot \seq \emptyset$). Moreover, as in the proof of Theorem \ref{thm:interpolation}, each variable in the separating set must occur both in the premises and in the conclusion of the proof. The separating set therefore again jointly plays the role of the $(\K, \LP)$-interpolant.
\end{proof}

  It is easy to see that for $\CL$, $\LP$, $\K$, and $\ETL$ the interpolation theorems above are optimal in a natural sense. Let $\logic{L}_{1}$ and $\logic{L}_{2}$ be super-Belnap logics.

\begin{proposition}
   If $\CL$ enjoys $(\logic{L}_{1}, \logic{L}_{2})$-interpolation, then $\logic{L}_{1} \supseteq \K$ and $\logic{L}_{2} \supseteq \LP$. If $\LP$ enjoys $(\logic{L}_{1}, \logic{L}_{2})$-interpolation, then $\logic{L}_{2} = \LP$. If $\K$ ($\ETL$) enjoys $(\logic{L}_{1}, \logic{L}_{2})$-interpolation, then $\logic{L}_{1} = \K$ ($\logic{L}_{1} = \ETL$).
\end{proposition}

\begin{proof}
  The logics $\logic{L}_{1}$ and $\logic{L}_{2}$ are non-trivial in each case. Note that the rule $(p \wedge \dmneg p) \vee q \vdash q$ is equivalent over $\B$ to the rule $p \vee q, \dmneg q \vee r \vdash p \vee r$.

  Since $(p \wedge \dmneg p) \vee q \vdash_{\CL} q$, we have $(p \wedge \dmneg p) \vee q \vdash_{\logic{L}_{1}} \chi$ and $\chi \vdash_{\logic{L}_{2}} q$ for some $\chi$ which does not contain any variable other than $q$. Then clearly $\chi \dashv \vdash_{\B} q$, hence $(p \wedge \dmneg p) \vee q \vdash_{\logic{L}_{1}} q$ and $\logic{L}_{1} \supseteq \K$. Likewise, since $\emptyset \vdash_{\CL} p \vee \dmneg p$, we have $\emptyset \vdash_{\logic{L}_{1}} \chi$ and $\chi \vdash_{\logic{L}_{2}} p \vee \dmneg p$ for some $\chi$ which does not contain any variables. Then clearly $\chi \dashv \vdash_{\B} \top$, hence $\top \vdash_{\logic{L}_{2}} p \vee \dmneg p$ and $\logic{L}_{2} \supseteq \LP$. The same argument shows that $\logic{L}_{2} \supseteq \LP$ in the case of $\LP$.

  Finally, since $(p \wedge \dmneg p) \vee q \vdash_{\K} q$, we have $(p \wedge \dmneg p) \vee q \vdash_{\K} q$, we have $(p \wedge \dmneg p) \vee q \vdash_{\logic{L}_{1}} \chi$ and $\chi \vdash_{\logic{L}_{2}} q$ for some $\chi$ which does not contain any variable other than $q$. Then clearly $\chi \dashv \vdash_{\B} q$, hence $(p \wedge \dmneg p) \vee q \vdash_{\logic{L}_{1}} q$ and $\logic{L}_{1} \supseteq \K$. The argument for $\ETL$ is entirely analogical.
\end{proof}

  In conclusion, we hope that the above constitutes sufficient evidence that a systematic investigation of the family of super-Belnap logics, in addition to having some intrinsic interest of its own, may shed some new light even on well studied systems such as classical propositional logic.


\paragraph{Acknowledgements.} 
This project has received funding from the European Union’s Horizon 2020 research and innovation programme under the Marie Skłodowska-Curie grant agreement No 689176. The author also gratefully acknowledges the support of the grant P202/12/G061 of the Czech Science Foundation. The author is grateful to Ori Lahav and Kazushige Terui for directing his attention to the relevant literature.


\end{document}